\documentclass[a4paper,11pt]{amsart}

\usepackage[applemac]{inputenc}
\usepackage{amsfonts} 
\usepackage{amsmath} 
\usepackage{amsthm} 
\usepackage{xcolor} 
\usepackage{hyperref}
\usepackage{amssymb}

\makeatletter
\@namedef{subjclassname@2020}{\textup{2020} Mathematics Subject Classification}
\makeatother

\title[Geometric effects of hyperbolic  classes on K\"ahler manifolds]{Geometric effects of hyperbolic cohomology classes on K\"ahler manifolds \\ (with an appendix by Benoît Claudon)}
\author[F. Bei, S. Diverio and S. Trapani]{Francesco Bei \and Simone Diverio  \and Stefano Trapani}
\address{Francesco Bei and Simone Diverio \\ Dipartimento di Matematica Guido Castelnuovo \\ Sapienza Universit\`a di Roma \\ Piazzale Aldo Moro 5 \\ I-00185 Roma.}
\email{bei@mat.uniroma1.it \\ simone.diverio@uniroma1.it } 
\address{Stefano Trapani\\ Dipartimento di Matematica Universit\`a di Roma \lq\lq Tor Vergata\rq\rq, Via della Ricerca Scientifica 1, 00133 Roma.} 
\email{trapani@mat.uniroma2.it} 
\thanks{The first-named author was partially supported by 2024 Sapienza research grant \lq\lq New research trends in Mathematics at Castelnuovo\rq\rq{} and by INdAM - GNSAGA Project, codice CUP E53C24001950001.\newline\indent
The second-named author was partially supported by 2024 Sapienza research grant \lq\lq New research trends in Mathematics at Castelnuovo\rq\rq{}, by 2023 Sapienza research grant \lq\lq Hyperk\"ahler varieties, deformation theory and hyperbolicity\rq\rq{}, and by the 2022 PRIN project \lq\lq Moduli spaces and special varieties\rq\rq{}.\newline\indent
The third-named author was partially supported by GNSAGA of Indam, by the excellence projects of the department of Mathematics of the University of Rome \lq\lq Tor Vergata\rq\rq{} 2018--2022 CUP E83C18000100006, 2023--2027 CUP E83C23000330006, and by PRIN \lq\lq Real and Complex Manifolds: Topology, Geometry and holomorphic dynamics\rq\rq{} n° 2017JZ2SW5.}
\date{\today}

\theoremstyle{plain}
\newtheorem{thm}{Theorem}[section]
\newtheorem{cor}[thm]{Corollary}
\newtheorem{lem}[thm]{Lemma}
\newtheorem{prop}[thm]{Proposition}
\newtheorem{quest}[thm]{Question}
\newtheorem{conj}[thm]{Conjecture}

\theoremstyle{remark}
\newtheorem{rem}[thm]{Remark}

\theoremstyle{definition}
\newtheorem{defn}[thm]{Definition}

\newcommand{\Z}{\mathbb{Z}}

\newcommand{\R}{\mathbb{R}}

\DeclareMathOperator{\dvol}{dvol}
\DeclareMathOperator{\im}{im}
\DeclareMathOperator{\vol}{vol}

\DeclareMathOperator{\hyp}{hyp}

\DeclareMathOperator{\Lip}{Lip}
\newcommand{\To}{\longrightarrow}
\newcommand{\linf}{\ell^{\infty}}
\newcommand{\Linf}[1]{\ell^{\infty}\left(#1\right)}
\newcommand{\Supp}[1]{\mathrm{Supp}\left(#1\right)}
\newcommand{\compb}{c_b}
\newcommand{\compinfty}{c_{(\infty)}}

\begin{document}
\begin{abstract}
We introduce the notion of Kähler topologically hyperbolic manifold, as a \lq\lq topological\rq\rq{} generalization of Kähler \cite{Gro91} and weakly Kähler \cite{BDET} hyperbolic manifolds. 
Analogously to \cite{BCDT}, we show the birational invariance of this property and then that Kähler topologically hyperbolic manifolds are not uniruled nor bimeromorphic to compact Kähler manifolds with trivial first real Chern class. 
Then, we prove spectral gap theorems for positive holomorphic Hermitian vector bundles on Kähler topologically hyperbolic manifolds, obtaining in particular effective non vanishing results \textsl{à la} Kawamata for adjoint line bundles. 
We finally explore the effects of Kähler topologically hyperbolicity on Ricci and scalar curvature of Kähler metrics.

In the appendix, it is given an explicit description of degree~$2$ hyperbolic classes for finitely presented groups, and an algebro-geometric consequence for Kähler topologically hyperbolic surfaces: they are necessarily of general type.
\end{abstract}
\bibliographystyle{alpha}
\maketitle

\section{Introduction}

In order to prove that every (possibly singular) subvariety of a K\"ahler hyperbolic manifold is of general type ---thus proving Lang's conjecture for this special class of Kobayashi hyperbolic manifolds--- the authors were led in \cite{BDET} to introduce a more general notion of K\"ahler hyperbolicity which they called \emph{weak K\"ahler hyperbolcity}. 

While a K\"ahler hyperbolic manifold  $M$ (in the sense of Gromov \cite{Gro91}) is a compact K\"ahler manifold admitting a K\"ahler form $\omega$ whose lifting to the universal cover becomes $d$-exact and moreover with a bounded primitive (this property is often referred to as $\tilde d$-boundedness), this weaker notion doesn't asks the form $\omega$ to be K\"ahler but rather the cohomology class $[\omega]\in H^{1,1}(M,\mathbb R)$ to be nef and big.

It turns out \cite{BDET} that such manifolds continue to enjoy several remarkable spectral, general typeness, as well as degeneracy of entire curves properties. Moreover, this notion seems to be the \lq\lq right\rq\rq{} birational version of K\"ahler hyperbolic manifolds \cite{BCDT}, as wished in \cite[Open Problem 18.7]{Kollarbook}.

In this paper, we address the following natural question: what if we drop the positivity (bigness and nefness) assumption on the cohomology class $[\omega]$ and we just leave as hypothesis (one of) the consequence of that positivity, namely having non vanishing top self-intersection (homological non singularity condition)? More than this, in order to make the definition as purely topological as possible, we also drop the hypothesis of being a $(1,1)$-class letting the class to be free to be any real $2$-cohomology class. This leads to define a further generalization, stable under homotopy equivalence in the realm of compact Kähler manifolds, of K\"ahler as well as weak K\"ahler hyperbolicity, which we shall call \emph{K\"ahler topological hyperbolicity}. Summing up, a compact K\"ahler manifold supporting a smooth closed $\tilde d$-bounded real $2$-form such that its top self-intersection has non-vanishing integral (cf. Definition \ref{def:thm}) will be called a K\"ahler topologically hyperbolic manifold. 

The point here is that in this business one of the crucial tools we have to control the spectral properties of the $L^2$ Hodge–Kodaira Laplacian is the so-called Gromov--Vafa--Witten trick, which needs only the top self intersection of the hyperbolic cohomology class to be non vanishing in order to work (cf. the proof of Proposition \ref{L2spec}).

First of all, likewise weak K\"ahler hyperbolicity, it turns out that this notion is bimeromorphically invariant for compact K\"ahler manifolds (cf. Theorem \ref{thm:bir}). As a consequence, we show that being K\"ahler topologically hyperbolic still has some weak positivity implications on the birational geometry of the manifold: namely, we prove that neither such a compact K\"ahler manifold can be dominated meromorphically by a product $\mathbb P^1\times V$ (cf. Proposition \ref{uniruled}), \textsl{i.e.} it cannot be uniruled, nor it can be bimeromorphic to a compact K\"ahler manifold with trivial first real Chern class (cf. Theorem \ref{thm:c1}). 
If the manifold is moreover projective (which for the moment does not seem to follow from the definition), the former implies by \cite{BDPP13} that the canonical class must be pseudoeffective, as stated in Corollary \ref{cor:pseff} (which conjecturally should be equivalent to having non-negative Kodaira dimension), while the latter morally and conjecturally should mean that the Kodaira dimension is different from zero. In the special case of compact Kähler surfaces, it will be moreover shown in the appendix that Kähler topological hyperbolicity implies general typeness.

\medskip

Next, in order to hope to be able to deduce any spectral consequence, we need to move the positivity assumptions we formerly had on the cohomology class elsewhere. Thus we investigate holomorphic vector bundles which are positive in the sense of Nakano on an open set of full mass and obtain the following effective non-vanishing statement.

\begin{thm}[Cf. Theorem \ref{nsp}]
Let $(M,h)$ be a K\"ahler topologically hyperbolic manifold. Let $(E,\tau)\rightarrow M$ be a Hermitian holomorphic vector bundle, Nakano positive over an open subset $A\subset M$ of full measure. Then,  
$$
H^0(M,K_M\otimes E)\neq \{0\}
$$ 
and  
$$
h^{0}(M,K_M\otimes E)=\chi(M,K_M\otimes E)>0.
$$
\end{thm}

In the particular case where $E$ is a holomorphic line bundle, the hypothesis on the positivity of the curvature implies that $E$ is big and nef, and thus the non-vanishing for the space of global holomorphic sections of the adjoint bundle $K_M\otimes E$ might be interpreted as a (very) special case of the Ambro--Ionescu--Kawamata effective non-vanishing conjecture (cf. Corollary \ref{line} and the subsequent Remark \ref{rem:aikconj}, also for comparisons with previously known results).

\medskip

Other instances that are affected by topological hyperbolicity are the scalar and Ricci curvatures of K\"ahler metrics on such manifolds. To be more reader friendly here in the introduction, we mention a typical results one can obtain under the more restrictive hypothesis of weak K\"ahler hyperbolicity, and we refer to Section \ref{sec:curv} for several other statements under the mere assumption of topological hyperbolicity.

\begin{thm}[Cf. Theorem \ref{scalcurv}]
Let $M$ be a weakly K\"ahler hyperbolic manifold. Then, for any arbitrarily fixed K\"ahler metric $h$, we have 
$$
\min_{M}(\mathrm{scal}_h)\leq -4\tilde{\lambda}_{0,h},
$$ 
and the equality holds if and only if  
$$
\mathrm{scal}_h\equiv -4\tilde{\lambda}_{0,h}.
$$
\end{thm}
Here, $\tilde{\lambda}_{0,h}$ stands for the bottom of the spectrum of the (unique $L^2$ closed extension of the) Laplace--Beltrami operator associated to the induced K\"ahler metric $\tilde h$ on the universal cover of $M$. 

In particular (cf. Corollary \ref{cor:noKm}), we see that such an $M$ cannot carry any K\"ahler metric with $\mathrm{scal}_h>-4\tilde{\lambda}_{0,h}$. In this situation, we already knew, essentially by \cite{BDPP13}, that a weak K\"ahler hyperbolic manifold cannot carry any K\"ahler metric whose total scalar curvature is non-negative (cf. Remark \ref{rem:bdpp}). We are thus giving a more precise quantitative version which tells us that not only the scalar curvature should average more negatively than positively but also that is must be \lq\lq negative enough\rq\rq{} somewhere, in terms of the bottom of the spectrum of the Laplace--Beltrami operator.

\medskip

Finally, the main technical tool of the paper is a spectral result which is proved in Section \ref{sec:ellopgalois}, and can be explained as follows. Suppose we have a differential formally self-adjoint elliptic operator $P$ acting on the smooth global sections of a complex Hermitian vector bundle on a Riemannian manifold, and suppose it admits a decomposition as $D+L$ where $D$ has the same properties as $P$ and $L$ is of order zero. Now, lift everything to the universal cover of $M$ and take the unique $L^2$ closed extension of the lifted operators. Then, we prove in Theorem \ref{pspec} the positivity of the spectrum of the lifting $\tilde P$ of $P$, provided the lifting $\tilde D$ of $D$ is non-negative and the negativity of $L$ is bounded below by minus the bottom of the spectrum of $\tilde D$, on an open set of full mass.

To finish with, unfortunately, it is honest to say that we don't dispose for the moment of explicit remarkable examples of K\"ahler topologically hyperbolic manifolds which are not already weakly K\"ahler hyperbolic. Though such examples are urgently needed, we anyway emphasize that our results are already new in the K\"ahler as well as the weakly K\"ahler hyperbolic settings. 

\subsubsection*{Acknowledgements}
The first-named author would like to thank Paolo Piazza for useful discussions on the local index theorem for spin-c operators.

The second-named author would like to thank S\'ebastien Boucksom and Andreas H\"oring for useful discussions about the effective non-vanishing conjecture by Ambro--Ionescu--Kawamata.

The authors warmly thanks Benoît Claudon for several remarks communicated almost instantaneously right after the first version of this paper has appeared on the arXiv.

\section{Homologically non singular and hyperbolic cohomology classes}

We start with some definitions and properties that will play a crucial role in rest of the paper. 
Let us consider a compact, oriented smooth manifold $N$ of dimension $n$ endowed with a Riemannian metric $h$. 

Following \cite{Kotschick}, we say that a cohomology class $[\eta]\in H^k_{\mathrm{dR}}(N)\simeq H^k(N,\mathbb R)$ is \emph{hyperbolic} if there exists $\beta\in L^{\infty}\Omega^{k-1}(\tilde{N},\tilde{h})\cap \Omega^{k-1}(\tilde{N})$ such that $\tilde{d}_{k-1}\beta=\pi^*\eta$, where $\pi\colon \tilde{N}\rightarrow N$ is the universal covering of $N$ with induced Riemannian metric $\tilde h=\pi^*h$. It is immediate to check that the definition does not depend on the representative of $[\eta]$ nor on the metric $h$, and that $V^k_{\mathrm{hyp}}(N)$, the set of hyperbolic $k$-cohomology classes, is a  vector subspace of $H^k(N,\mathbb R)$. 

If $n=k \ell$ for some $\ell\in \mathbb{N}$ we introduce the following subset of \emph{homologically non singular} classes
$$
V^k_{\mathrm{hns}}(N)\subset H^k(N,\mathbb R)
$$ 
as 
$$
V^k_{\mathrm{hns}}(N):=\left\{[\eta]\in H^k(N,\mathbb R)\mid  \int_N[\eta^{\ell}]\neq 0\right\}.
$$
Note that:
\begin{itemize}
\item if $k=n$ then $V^{n}_{\mathrm{hns}}(N)=H^n(N,\mathbb R)\setminus\{0\}$;
\item if $V^{n}_{\mathrm{hyp}}(N)\neq \{0\}$ then $\pi_1(N)$ is infinite and $V^{n}_{\mathrm{hyp}}(N)=H^{n}_{\mathrm{dR}}(N)$;
\item if $n=k\ell$ with $\ell\in \mathbb{N}$ and $V^k_{\mathrm{hns}}(N)\cap V^k_{\mathrm{hyp}}(N)\neq \emptyset$ then we have
 $V^{n}_{\mathrm{hns}}(N)=V^{n}_{\mathrm{hyp}}(N)\setminus \{0\}=H^n(N,\mathbb R)\setminus \{0\}$ and, in particular, $\pi_1(N)$ is infinite.
\end{itemize}

Let us single out the case $k=2$ on a compact Kähler manifold, which is particularly significant for us.
\begin{defn}\label{def:thm}
A compact K\"ahler manifold $M$ with 
$$
V^{2}_{\mathrm{hns}}(M)\cap V^{2}_{\mathrm{hyp}}(M)\neq \emptyset
$$ 
will be called \emph{Kähler topologically hyperbolic}.
\end{defn}

Examples of Kähler topologically hyperbolic manifolds are provided for instance by K\"ahler hyperbolic manifolds and more generally by weakly K\"ahler hyperbolic manifolds. 

We recall that a compact K\"ahler manifold $M$ is {\em K\"ahler hyperbolic} if there exists a K\"ahler form $\omega$ such that $[\omega]\in V^{2}_{\mathrm{hyp}}(M)$ whereas it is {\em weakly K\"ahler hyperbolic} if there exists $[\mu]\in V^{2}_{\mathrm{hyp}}(M)\cap H^{1,1}(M,\mathbb{R})$ which is big and nef, see \cite{Gro91} and \cite{BDET}, respectively.

\section{Background on K\"ahler manifolds and $L^2$ Hodge theory}
Let $(M,\omega)$ be a complete K\"ahler manifold of complex dimension $m$, let $h$ be the corresponding K\"ahler metric on $M$ and let $(E,\tau)\rightarrow M$ be a holomorphic vector bundle over $M$. We denote with $\overline{\partial}_{E,p,q}\colon \Omega^{p,q}(M,E)\rightarrow \Omega^{p,q+1}(M,E)$ the Dolbeault operator acting on $E$-valued $(p,q)$-forms  and with   
 $\overline{\partial}_{E, p,q}^t\colon \Omega^{p,q+1}(M,E)\rightarrow \Omega^{p,q}(M,E)$  its  formal adjoint with respect to the metric $h$.  The Hodge--Kodaira Laplacian acting on $E$-valued $(p,q)$-forms, denoted here with $\Delta_{\overline{\partial}_{E},p,q}\colon \Omega^{p,q}(M,E)\rightarrow \Omega^{p,q}(M,E)$, is defined as 
 $$
 \Delta_{\overline{\partial}_{E},p,q}:= \overline{\partial}_{E,p,q}^t\circ \overline{\partial}_{E,p,q}+\overline{\partial}_{E,p,q-1}\circ \overline{\partial}_{E,p,q-1}^t.
 $$ 
 Let $L^2\Omega^{p,q}(M,E)$ be the Hilbert space of measurable and square integrable, $E$-valued, $(p,q)$-forms.  We look now at 
\begin{equation}
\label{terzo}
\Delta_{\overline{\partial}_{E},p,q}\colon  L^2\Omega^{p,q}(M,E)\rightarrow L^2\Omega^{p,q}(M,E)
\end{equation}
as an unbounded, closable and densely defined operator on $\Omega^{p,q}_c(M,E)$. It is well known that \eqref{terzo} is essentially self-adjoint, see \textsl{e.g.} \cite[Prop. 12.2]{BDIP}.  Since \eqref{terzo} is formally self-adjoint this is equivalent to saying that \eqref{terzo} admits a unique closed extension. Henceforth, with a little abuse of notation, we denote with
\begin{equation}
\label{sesto}
\Delta_{\overline{\partial}_{E},p,q}\colon  L^2\Omega^{p,q}(M,E)\rightarrow L^2\Omega^{p,q}(M,E)
\end{equation}
the unique closed (and hence self-adjoint) extension of \eqref{terzo}. Since $(M,h)$ is complete \eqref{sesto} is simply 
$$
\overline{\partial}_{E,p,q}^*\circ \overline{\partial}_{E,p,q}+\overline{\partial}_{E,p,q-1}\circ \overline{\partial}_{E,p,q-1}^*\colon  L^2\Omega^{p,q}(M,E)\rightarrow L^2\Omega^{p,q}(M,E)
$$ 
with 
$$
\overline{\partial}_{E,p,q}\colon  L^2\Omega^{p,q}(M,E)\rightarrow L^2\Omega^{p,q+1}(M,E)
$$ 
and 
$$ 
\overline{\partial}_{E,p,q-1}\colon  L^2\Omega^{p,q-1}(M,E)\rightarrow L^2\Omega^{p,q}(M,E)
$$ 
the unique closed extensions of $\overline{\partial}_{E,p,q}\colon  \Omega_c^{p,q}(M,E)\rightarrow \Omega_c^{p,q+1}(M,E)$ and $\overline{\partial}_{E,p,q-1}\colon  \Omega_c^{p,q-1}(M,E)\rightarrow \Omega_c^{p,q}(M,E)$, respectively and  
$$
\overline{\partial}_{E,p,q}^*\colon  L^2\Omega^{p,q+1}(M,E)\rightarrow L^2\Omega^{p,q}(M,E)
$$ 
and 
$$
\overline{\partial}_{E,p,q-1}^*\colon  L^2\Omega^{p,q}(M,E)\rightarrow L^2\Omega^{p,q-1}(M,E)
$$  
the corresponding adjoints. It follows immediately that in $L^2\Omega^{p,q}(M,E)$ we have
$$
\ker(\Delta_{\overline{\partial}_{E},p,q})=\ker(\overline{\partial}_{E,p,q})\cap \ker(\overline{\partial}_{E,p,q-1}^*)
$$ 
and 
$$
\overline{\text{im}(\Delta_{\overline{\partial}_{E},p,q})}=\overline{\text{im}(\overline{\partial}_{E,p,q-1})}\oplus \overline{\text{im}(\overline{\partial}^*_{E,p,q})}.
$$ 
Now we continue this section by recalling the definition and the basic properties of the $L^2$-Hodge numbers. We introduce only what is strictly necessary for our scopes with no goal of completeness.  We invite the reader to consult the seminal paper of Atiyah \cite{Atiyah} and the monograph \cite{Luck} for an in-depth treatment. 

Let $\pi\colon\tilde{M}\rightarrow M$ be a Galois covering of $M$ and let us denote with $\tilde{h}:=\pi^*h$, $\tilde{E}:=\pi^* E$, and $\tilde{\tau}:=\pi^*\tau$ the corresponding pullbacks. Let $\Gamma$ be the group of deck transformations acting on $\tilde{M}$, $\Gamma\times \tilde{M}\rightarrow \tilde{M}$. We recall that  $\Gamma$ is a discrete group acting fiberwise on $\tilde{M}$ and such that the action is transitive on each fiber and properly discontinuous. In particular we have $\tilde{M}/\Gamma=M$. 
Moreover the action of $\Gamma$ lift naturally on $\tilde{E}$ in such a way that $\tilde{E}/\Gamma=E$. 

We recall now that an open subset $U\subset \tilde{M}$ is a  fundamental domain of the action of $\Gamma$ on $\tilde{M}$ if
\begin{itemize}
\item $\tilde{M}=\bigcup_{\gamma\in \Gamma}\gamma(\overline{U})$, 
\item $\gamma_1(U)\cap\gamma_2(U)=\emptyset$ for every $\gamma_1, \gamma_2\in \Gamma$ with $\gamma_1\neq \gamma_2$,
\item $\overline{U}\setminus U$ has zero measure.
\end{itemize}   
It is not difficult to show that $L^2\Omega^{p,q}(\tilde{M},\tilde{E})\cong L^2\Gamma \otimes L^2\Omega^{p,q}(U, E|_U)\cong L^2\Gamma\otimes L^2\Omega^{p,q}(M,E)$ where the tensor products on the right are meant as tensor product of Hilbert spaces. Note that a basis for $L^2\Gamma$ is given by the functions $\delta_{\gamma}$ with $\gamma\in \Gamma$ defined as $\delta_{\gamma}(\gamma')=1$ if $\gamma=\gamma'$ and $\delta_{\gamma}(\gamma')=0$ if $\gamma\neq \gamma'$. Moreover it is clear that $\Gamma$ acts on $L^2\Omega^{p,q}(\tilde{M},\tilde{E})$ by isometries. Let us consider a $\Gamma$-module $V\subset L^2\Omega^{p,q}(\tilde{M},\tilde{E})$, that is a closed subspace of $L^2\Omega^{p,q}(\tilde{M},\tilde{E})$ which is invariant under the action of $\Gamma$. If $\{s_j\}_{j\in \mathbb{N}}$ is an orthonormal basis for $V$ then the following (possibly divergent) series 
$$
\sum_{j\in \mathbb{N}}\int_U\langle s_j,s_j\rangle_{\tilde{h}\otimes \tilde{\tau}}\dvol_{\tilde{h}}
$$ 
is well defined and does not depend either on the choice of the orthonormal basis of $V$ or on the choice of the fundamental domain of the action of $\Gamma$ on $\tilde{M}$, see \textsl{e.g.} \cite{Atiyah}. 

The Von Neumann dimension of a $\Gamma$-module $V$ is thus defined as 
$$
\dim_{\Gamma}(V):=\sum_{j\in \mathbb{N}}\int_U\langle s_j,s_j\rangle_{\tilde{h}\otimes \tilde{\tau}}\dvol_{\tilde{h}}
$$ 
where $\{\eta_j\}_{j\in \mathbb{N}}$ is any orthonormal basis for $V$ and $U$ is any fundamental domain of the action of $\Gamma$ on $\tilde{M}$.  Since the Hodge--Kodaira Laplacian $\tilde{\Delta}_{\overline{\partial}_{E},p,q}\colon L^2\Omega^{p,q}(\tilde{M},\tilde{E})\rightarrow L^2\Omega^{p,q}(\tilde{M},\tilde{E})$ commutes with the action of $\Gamma$, a natural and important example of $\Gamma$-module is provided by 
$\ker(\tilde{\Delta}_{\overline{\partial}_{\tilde{E}},p,q})$, the space of $\tilde{E}$-valued, $L^2$-harmonic forms of degree $(p,q)$ on $(\tilde{M},\tilde{E})$, for each $p,q=0,\dots,m$.  

The $L^2$-Hodge numbers of  $E\rightarrow M$ with respect to a Galois $\Gamma$-covering $\pi:\tilde{M}\rightarrow M$ are then defined as $$h^{p,q}_{(2),\Gamma,\overline{\partial}_E}(M):=\dim_{\Gamma}(\ker(\Delta_{\overline{\partial}_E,p,q}))$$
with $\ker(\Delta_{\overline{\partial}_E,p,q})$ the kernel of $\Delta_{\overline{\partial}_E,p,q}\colon L^2\Omega^{p,q}(\tilde{M},\tilde{E})\rightarrow L^2\Omega^{p,q}(\tilde{M},\tilde{E})$. We point out that $h^{p,q}_{(2),\Gamma,\overline{\partial}_E}(M)$ are non-negative real numbers independent on the choice of the Hermitian metrics $h$ on $M$ and $\tau$ on $E$. When the $L^2$-Hodge numbers are computed with respect to the universal covering of $M$ we will denote them as  $h^{p,q}_{(2),\overline{\partial}_E}(M,E)$.

\section{First properties and birational geometry of K\"ahler topologically hyperbolic manifolds}\label{sec:bimer}

Let us start with some basic properties of K\"ahler topologically hyperbolic manifolds.

\begin{prop}\label{prop:etalecover}
Let $M$ be a compact Kähler manifold and $\nu\colon M'\to M$ be a finite étale cover. Then, $M$ is Kähler topologically hyperbolic if and only if $M'$ is so.
\end{prop}

\begin{proof}
The proof is identical to that of \cite[Proposition 2.6]{BDET}.
\end{proof}

Likewise Kähler and weakly Kähler hyperbolic manifolds, Kähler topologically hyperbolic manifolds do not have amenable fundamental groups. 

\begin{prop}\label{prop:amenable}
Let $M$ be a Kähler topologically hyperbolic manifold. Then, $\pi_1(M)$ is not amenable. In particular, it has exponential growth for any finite system of generators and can be neither virtually abelian nor virtually nilpotent.
\end{prop}

\begin{proof}
Here also, the proof is identical to that of \cite[Proposition 2.7]{BDET}.
\end{proof}

We also have the following behaviour for products.

\begin{prop}\label{prop:prod}
The product of two Kähler topologically hyperbolic manifolds is again Kähler topologically hyperbolic.

Conversely, let $M$ be a compact K\"ahler manifold which is biholomorphic to a non trivial product $N\times V$, where $H^1(N,\mathbb R)=0$. 
If $M$ is K\"ahler topologically hyperbolic, so are $N$ and $V$.
\end{prop}

\begin{proof}
The fact that the product of two Kähler topologically hyperbolic manifolds is again Kähler topologically hyperbolic is straightforward. Indeed, if $\alpha$ and $\beta$ are real $2$-cohomology classes, respectively on the compact Kähler manifolds $M_1$ and $M_2$, which are hyperbolic and homologically non singular, then $\operatorname{pr_1}^*\alpha+\operatorname{pr_2}^*\beta$ is a real $2$-cohomology class on $M_1\times M_2$ which is hyperbolic and homologically non singular, where $\operatorname{pr_i}\colon M_1\times M_2\to M_i$ is the projection onto the $i$-th factor.

For the converse, identifying $M$ with its biholomorphic image $N\times V$, call respectively $\operatorname{pr}_N\colon M\to N$ and $\operatorname{pr}_V\colon M\to V$ the projections onto the factors. By K\"unnet formula, we have 
$$
H^2(M,\mathbb R)\simeq \operatorname{pr}_N^*H^2(N,\mathbb R)\oplus\operatorname{pr}_V^*H^2(V,\mathbb R),
$$
since $H^1(N,\mathbb R)=0$. Given $\alpha\in H^2(M,\mathbb R)$, we accordingly have $\alpha=\operatorname{pr}_N^*\beta+\operatorname{pr}_V^*\gamma$ for some classes $\beta\in H^2(N,\mathbb R)$ and $\gamma\in H^2(V,\mathbb R)$. Fix any point $(p_0,q_0)\in N\times V$ and consider the inclusions 
$$
\iota_1\colon N\simeq N\times\{q_0\}\hookrightarrow N\times V,\quad\iota_2\colon V\simeq \{p_0\}\times V\hookrightarrow N\times V.
$$
We have $\operatorname{pr}_N\circ\iota_1=\operatorname{Id}_N$ and $\operatorname{pr}_V\circ\iota_2=\operatorname{Id}_V$, so that necessarily $\beta=\iota_1^*\alpha$ and $\gamma=\iota_2^*\alpha$. In particular, by \cite[Lemma 2.28]{BDET} if $\alpha$ is a hyperbolic class, so are $\beta$ and $\gamma$.

Set $0<2n=\dim N,2v=\dim V$, so that $m=n+v$. If $\alpha$ is also homologically non singular, then 
$$
0\ne \alpha^{m}=\sum_{i=0}^m\binom{m}{i}\operatorname{pr}_N^*\beta^{m-i}\wedge\operatorname{pr}_V^*\gamma^{i}=\binom{m}{v}\operatorname{pr}_N^*\beta^{n}\wedge\operatorname{pr}_V^*\gamma^{v},
$$
for trivial dimensional reasons, so that $\beta^{n},\gamma^v\ne 0$ are homologically non singular and the lemma follows.
\end{proof}

\begin{rem}
Without the hypothesis about the vanishing of $H^1$, in general we have the presence of the factor $\operatorname{pr}_N^*H^1(N,\mathbb R)\otimes\operatorname{pr}_V^*H^1(V,\mathbb R)$ in the K\"unnet decomposition of $H^2(M,\mathbb R)$ which prevents us to be able to conclude with the same proof. 

It is quite surprising and non trivial that such a factor does not appear if one restricts the attention to hyperbolic classes. It is indeed proven by B. Claudon in Appendix \ref{sec:appendix}, Propositions \ref{prop:hyp product} and \ref{prop:explicit description}, that 
$$
V^2_\textrm{hyp}(M)\simeq\operatorname{pr}_N^*V^2_\textrm{hyp}(N)\oplus\operatorname{pr}_V^*V^2_\textrm{hyp}(V).
$$
This implies in particular that our Proposition \ref{prop:prod} above in fact holds without any supplementary hypothesis about the vanishing of the first cohomology of the factors and thus reads: \emph{The product of two compact Kähler manifolds is Kähler topologically hyperbolic if and only if both factors are}.
\end{rem}

Last but not least, as announced in the introduction, being Kähler topologically hyperbolic is a topological notion, in the sense that it is a stable property with respect to homotopy equivalence in the realm of compact Kähler manifolds.

\begin{prop}\label{prop:homotopic}
Let $M$ be a compact Kähler manifold which is homotopy equivalent to a Kähler topologically hyperbolic manifold $X$. Then, $M$ is topologically Kähler hyperbolic.
\end{prop} 

\begin{proof}
Let $f\colon M\to X$ and $g\colon X\to M$ be two continuous maps realizing the homotopy equivalence. By the Whitney Approximation Theorem we can suppose that $f,g$ are smooth and that $f\circ g,g\circ f$ are smoothly homotopic to the respective identity maps. 

Take any $\alpha\in V^2_\textrm{hyp}(X)\cap V^2_\textrm{hns}(X)$. Then, $\beta:=f^*\alpha\in V^2_\textrm{hyp}(M)$ by \cite[Lemma 2.28]{BDET}. Moreover, $f^*\colon H^\bullet(X,\mathbb R)\to H^\bullet(M,\mathbb R)$ is an isomorphism with inverse map $g^*\colon H^\bullet(M,\mathbb R)\to H^\bullet(X,\mathbb R)$. In particular, if we let $\dim M=m$ and $\dim X=n$, it follows that $m=n$ because $H^{2m}(M,\mathbb R)\simeq\mathbb R\ne\{0\}$ implies $H^{2m}(X,\mathbb R)\simeq\mathbb R\ne\{0\}$ so that $\dim X\ge\dim M$, and we also get the reverse inequality since the situation is symmetric. Then, $\beta^m\ne 0$ if and only if $g^*\beta^m\ne 0$, and $g^*\beta^m=g^*f^*\alpha^m=\alpha^m\ne 0$ being $\alpha\in V^2_\textrm{hns}(X)$. Thus, $\beta\in V^2_\textrm{hns}(M)$ and $M$ is Kähler topologically hyperbolic.
\end{proof}

We now pass to birational properties and show that the property of being topologically hyperbolic for a compact K\"ahler manifold is a stable property with respect to bimeromorphic mappings. More precisely we have the next

\begin{thm}\label{thm:bir}
Let $M$ and $N$ be two compact bimeromorphic K\"ahler manifolds. Then, $M$ is Kähler topologically hyperbolic if and only if $N$ is so.
\end{thm}

\begin{proof}
Let $\psi\colon M\dasharrow N$ be a bimeromorphic mapping. Taking a resolution of the closure of its graph we end up with a compact K\"ahler manifold $V$ together with two modifications $p_M\colon V\rightarrow M$ and $p_N\colon V\rightarrow N$ such that $\psi\circ p_M=p_N$. Therefore, in order to prove the above statement it is enough to consider the case of a modification $\psi\colon M\rightarrow N$. Clearly if $[\beta]\in V^2_{\mathrm{hyp}}(N)\cap V^2_{\mathrm{hns}}(N)$ then $[\psi^*\beta]\in V^2_{\mathrm{hyp}}(M)\cap V^2_{\mathrm{hns}}(M)$. 

The converse implication is subtler and relies on the same topological tool used to prove the main theorem in \cite{BCDT}. Let us denote with $\Gamma$ the fundamental group of $N$ and let $\tilde{N}\rightarrow N$ be a representative of the universal covering of $N$. Note that this is a principal $\Gamma$-bundle. Let now $c\colon N\rightarrow B\Gamma$ be the corresponding classifying map. Then, given that $\psi\colon M\rightarrow N$ is a modification, we know that $\psi^*\tilde{N}\rightarrow M$ is principal $\Gamma$-bundle that is also connected and simply connected and thus is a representative of the universal covering of $M$. Moreover its classifying map is given by $b:=c\circ \psi$. 

Let us consider now al class $[\alpha]\in  V^2_{\mathrm{hyp}}(M)\cap V^2_{\mathrm{hns}}(M)$. Then, according to \cite[Cor. 2.8]{BCDT}, there exists a hyperbolic cohomology class of degree two  $[\upsilon]\in H^2(B\Gamma,\mathbb{R})$, such that $b^*[\upsilon]=[\alpha]$. Since $b\colon M\rightarrow B\Gamma$ is a classifying map we can use \cite[Th. 2.4]{Kotschick} to conclude that $c^*[\upsilon]\in V^2_{\mathrm{hyp}}(N)$. Finally, by the fact that $\psi^*(c^*[\upsilon])=[\alpha]$ and that $\psi\colon M\rightarrow N$ is a modification, we can conclude that $c^*[\upsilon]\in  V^2_{\mathrm{hyp}}(N)\cap V^2_{\mathrm{hns}}(N)$.
\end{proof}

Following the arguments used in \cite[Cor. 1.3]{BCDT} we have also the following

\begin{cor}\label{cor:dominant}
Let $f\colon M\dasharrow N$ be a generically finite dominant meromorphic mapping between compact K\"ahler manifolds.

If $N$ is K\"ahler topologically hyperbolic, then so is $M$. In particular this holds if $f\colon M\dasharrow N$ is a dominant meromorphic mapping between compact K\"ahler manifolds of the same dimension.

Conversely, if $M$ is K\"ahler topologically hyperbolic and the induced map on the fundamental groups is injective, then $N$ is K\"ahler topologically hyperbolic.
\end{cor}

\begin{proof}
We leaves the details to the reader as this follows by an immediate adaptation of the arguments used to prove \cite[Cor. 1.3]{BCDT}. We just reproduce here the proof of the fact that if $f\colon M\dasharrow N$ is a dominant meromorphic mapping between compact K\"ahler manifolds of the same dimension, then $M$ is K\"ahler topologically hyperbolic if $N$ is so, since it will be used below.

By desingularizing the closure of graph of $f$ in $N\times M$ we get a compact smooth K\"ahler manifold $L$ with a modification $p_N\colon L\rightarrow N$ and and a holomorphic map $p_M\colon L\rightarrow M$ such that $p_M=f\circ p_N$. Since $L$ and $M$ have the same dimension and $f$ is dominant we can deduce that $p_M$ is surjective and thus $p_M$ has non-trivial degree. We thus obtain that if $[\eta] \in V^{2}_{\mathrm{hyp}}(M)\cap V^{2}_{\mathrm{hns}}(M)$ then $[p^*_M\eta]\in V^{2}_{\mathrm{hyp}}(L)\cap V^{2}_{\mathrm{hns}}(L)$ and so $L$ is a Kähler topologically hyperbolic manifold. Finally, given that $N$ and $L$ are bimeromorphic, we can conclude that $N$ is Kähler topologically hyperbolic, too.
\end{proof}

\begin{rem}
At this point it worths noticing how these different notions of hyperbolicity are related. We have of course that
\begin{multline*}
\textrm{$\{$Kähler hyperbolic$\}\subsetneq\{$weakly Kähler hyperbolic$\}$} \\ \subset\{\textrm{Kähler topologically hyperbolic$\}$}
\end{multline*}
and we strongly believe that also the last inclusion is strict.

Now, given any of these three set, call it $A$, denote by $\overline A^\textrm{bir}$ and $\overline A^\textrm{hom}$ respectively the set of compact Kähler manifolds which are bimeromorphic to some manifold in $A$, and the set of compact Kähler manifolds which are homotopy equivalent to some manifold in $A$.

We then have 
\begin{multline*}
\{\textrm{Kähler hyperbolic}\}\subsetneq\overline{\{\textrm{Kähler hyperbolic}\}}^\textrm{bir} \\
\subsetneq\overline{\{\textrm{weakly Kähler hyperbolic}\}}^\textrm{bir}=\{\textrm{weakly Kähler hyperbolic}\}\subset \\ 
\{\textrm{Kähler topologically hyperbolic}\}=\overline{\{\textrm{Kähler topologically hyperbolic}\}}^\textrm{bir}.
\end{multline*}

The first strict inclusion is immediate and already observed in \cite{BDET}, while the second strict inclusion is far from being obvious and is shown in \cite{BCDT}.

Finally we have,
\begin{multline*}
\{\textrm{Kähler hyperbolic}\}\subset\overline{\{\textrm{Kähler hyperbolic}\}}^\textrm{hom} \\
\subsetneq\overline{\{\textrm{weakly Kähler hyperbolic}\}}^\textrm{hom}\subset \overline{\{\textrm{Kähler topologically hyperbolic}\}}^\textrm{hom}\\ 
=\{\textrm{Kähler topologically hyperbolic}\}.
\end{multline*}
To justify the only strict inclusion above, take a Kähler hyperbolic manifold of dimension greater than one and blow it up along smooth submanifolds. It will be clearly weakly Kähler hyperbolic.
Now, on the one hand the sign of the topological Euler characteristic of a Kähler hyperbolic manifold of dimension $n$ is $(-1)^n$ \cite[Theorem 0.4.A, Remark 0.4.B]{Gro91}. On the other hand one can always make the topological Euler characteristic positive by blowing up points, while one can always make it negative by blowing up curves of large genus (provided $n\ge 3$) \cite[Theorem 7.31]{Voi07}.
Since homotopy equivalent compact Kähler manifolds have the same topological Euler characteristic, those blow ups give the strict inclusion in any dimension $n\ge 3$.  

We don't know at the moment whether or not the other inclusions are strict, even if we strongly think this is the case at least for the former.
\end{rem}

Next, we use the above results to study the positivity of the canonical class of a K\"ahler topological hyperbolic manifold. In particular, we get the non-uniruledness of K\"ahler topologically hyperbolic manifolds, in the following sense. 

\begin{prop}
\label{uniruled}
Let $f\colon\mathbb{P}^1\times N\dasharrow M$ be a dominant meromorphic map where $N$ and $M$ are compact K\"ahler manifolds of dimension $m-1$ and $m$ respectively. Then, $M$ is not K\"ahler topologically hyperbolic.
\end{prop}

\begin{proof}
In view of Corollary \ref{cor:dominant}, we need to show that $X=\mathbb{P}^1\times N$ is not K\"ahler topologically hyperbolic. 
But this follows at once from Proposition \ref{prop:prod}, since $\mathbb P^1$ is simply connected and obviously not K\"ahler topologically hyperbolic.
\end{proof}

\begin{cor}\label{cor:pseff}
Let $M$ be a projective K\"ahler topologically hyperbolic manifold. 
Then, $M$ is not uniruled and thus $K_M$ is pseudoeffective.
\end{cor}

\begin{proof}
The fact that $M$ is not uniruled is just a rephrasing of Proposition \ref{uniruled}. By \cite{BDPP13}, we thus have that $K_M$ is pseudoeffective.
\end{proof}

We finally prove that a K\"ahler topologically hyperbolic manifold cannot be bimeromorphic to compact Kähler manifold with trivial first real Chern class, which --by standard conjectures in the Minimal Model Program--- is not so far from having non zero Kodaira dimension.

\begin{thm}\label{thm:c1}
Let $M$ be a K\"ahler topologically hyperbolic manifold. Then, $M$ cannot be bimeromorphic to any compact Kähler manifold $X$ whose $c_1(X)\in H^2(X,\mathbb R)$ is zero.
\end{thm}

\begin{proof}
Thanks to Theorem \ref{thm:bir}, it suffices to show that such an $X$ cannot be K\"ahler topologically hyperbolic. By the celebrated Beauville--Bogomolov decomposition theorem we have that a finite \'etale cover $X'$ of $X$ is biholomorphic to a product $T\times\operatorname{CY}\times\operatorname{HK}$, where $\operatorname{CY}$ and $\operatorname{HK}$ are respectively products of \textsl{stricto sensu} Calabi--Yau manifolds and hyperK\"ahler manifolds. In particular, $\operatorname{CY}\times\operatorname{HK}$ is simply connected and compact. 

Therefore, the fundamental group of $X$ is virtually abelian, so that $X$ cannot be Kähler topologically hyperbolic thanks to Proposition \ref{prop:amenable}.
\end{proof}

\begin{rem}
Another way to conclude is the following. If $\operatorname{CY}\times\operatorname{HK}$ is positive dimensional, $X'$ is not K\"ahler topologically hyperbolic by Proposition \ref{prop:prod}. Otherwise, $X'$ is a complex torus which is not K\"ahler topologically hyperbolic since the Laplace--Beltrami operator (with respect to any metric) on its universal cover always has zero in its spectrum. This is not possibile for K\"ahler topologically hyperbolic manifolds since on their universal cover the Cheeger inequality always holds, so that the Laplace--Beltrami operator has entirely positive spectrum (see Proposition \ref{C-In} below).

Either ways, $X'$ is not K\"ahler topologically hyperbolic, and so $X$ is not K\"ahler topologically hyperbolic by Proposition \ref{prop:etalecover}.
\end{rem}

The above results together with Campana's Abelianity Conjecture \cite{Cam04,Cam11}, naturally lead to conjecture the following.

\begin{conj}
A special manifold in the sense of Campana is never K\"ahler topologically hyperbolic.
\end{conj}

\section{Elliptic operators and positivity of the spectrum on Galois coverings}\label{sec:ellopgalois}
In this section we develop the necessary tools of spectral theory that will play a crucial role in the rest of the paper. Likewise in \cite[Th. 3.1]{BDET} we aim to find a condition that rules out zero in the spectrum of certain differential operators. The main difference with \cite{BDET} is that here we address the above problem for a large class of elliptic differential operators with no restriction on the order and that admit a certain decomposition in terms of a linear potential, see \eqref{dec} below. 

Let us start with some definitions. 
Let $(M,g)$ be a compact Riemannian manifold with infinite fundamental group. Let $(E,\rho)\rightarrow M$ be a complex Hermitian vector bundle and let $P:C^{\infty}(M,E)\rightarrow C^{\infty}(M,E)$ be a formally self-adjoint  elliptic differential operator of order $d$. Let us assume that 
\begin{equation}
\label{dec}
P=D+L
\end{equation}
 with $D:C^{\infty}(M,E)\rightarrow C^{\infty}(M,E)$ another formally self-adjoint elliptic differential operator of order $d$ and $L\in C^{\infty}(M,\mathrm{End}(E))$.  Let $\pi:\tilde{M}\rightarrow M$ be a Galois covering with $\Gamma$ the corresponding group of deck transformations; set $\tilde{E}=\pi^*E$, $\tilde{P}$, $\tilde{D}$ and $\tilde{L}$ the lift of $P$, $D$ and $L$, respectively. Note that the group $\Gamma$ acts in a natural way on $\tilde{E}$. Indeed, given any $\gamma\in \Gamma$, $x\in \tilde{M}$ and  $(x,v)\in \tilde{E}_x$ with $v\in E_{\pi(x)}$, the action of $\gamma$ on $(x,v)$ is given by  $\gamma_E(x,v):=(\gamma(x),v)$. This gives clearly a linear isomorphism $\gamma_E:\tilde{E}_x\rightarrow \tilde{E}_{\gamma(x)}$ which is also an isometry with respect to $\tilde{\rho}$. Thus, the group $\Gamma$ also acts on $C^{\infty}(\tilde{M},\tilde{E})$  by $\gamma(s):=\gamma_E\circ s\circ \gamma^{-1}$. Note that $\gamma(\tilde{D}s)=\tilde{D}(\gamma(s))$ as well as $\gamma(\tilde{P}s)=\tilde{P}(\gamma(s))$ for any $s\in C^{\infty}(\tilde{M},\tilde{E})$. Let  $$
\tilde{D}\colon L^2(\tilde{M},\tilde{E},\tilde{g},\tilde{\rho})\rightarrow L^2(\tilde{M},\tilde{E},\tilde{g},\tilde{\rho})$$
be the unique $L^2$-closed extension of $\tilde{D}\colon C_c^{\infty}(\tilde{M},\tilde{E})\rightarrow C^{\infty}_c(\tilde{M},\tilde{E})$. Let us also consider $$\tilde{P}\colon L^2(\tilde{M},\tilde{E},\tilde{g},\tilde{\rho})\rightarrow L^2(\tilde{M},\tilde{E},\tilde{g},\tilde{\rho})$$ that is the unique $L^2$-closed extension of $\tilde{P}\colon C_c^{\infty}(\tilde{M},\tilde{E})\rightarrow C^{\infty}_c(\tilde{M},\tilde{E})$. We denote with $\sigma(\tilde{D})$ and $\sigma(\tilde{P})$ the spectrum of $\tilde{D}$ and $\tilde{P}$, respectively.

\begin{thm}
\label{pspec}
In the above setting assume that $\tilde{D}$ is non-negative and that $L+\min(\sigma(\tilde{D}))>0$ over $A$, an open subset of $M$ with $\vol_g(A)=\vol_g(M)$. Then $$0< \min(\sigma(\tilde{P})).$$
\end{thm}

In order to prove this theorem we need some preliminary results. 

\begin{prop}
\label{Normcontrol}
Let $(M,g)$, $(E,\rho)$ and $D\colon C^{\infty}(M,E)\rightarrow C^{\infty}(M,E)$ be as above. Let  
$$
\tilde{D}\colon L^2(\tilde{M},\tilde{E},\tilde{g},\tilde{\rho})\rightarrow L^2(\tilde{M},\tilde{E},\tilde{g},\tilde{\rho})
$$ 
be the unique $L^2$-closed extension of $\tilde{D}\colon C_c^{\infty}(\tilde{M},\tilde{E})\rightarrow C^{\infty}_c(\tilde{M},\tilde{E})$  and let $\{E(\lambda)\}_{\lambda}$ be its the spectral resolution. Fix $\lambda_0>0$. 

Then there exists $\varepsilon_0(\lambda_0)>0$ such that for any  open subset  $U$ of $M$ with 
$$
\vol_g(U)<\varepsilon_0(\lambda_0)
$$
we have:
$$
\int_{\tilde{U}}|s|^2_{\tilde{\rho}}\dvol_{\tilde{g}}\leq \int_{{\tilde{M}\setminus \tilde{U}}}|s|^2_{\tilde{\rho}}\dvol_{\tilde{g}},
$$
for all $s\in \im(E(\lambda_0))$ and where $\tilde{U}$ is the preimage of $U$ through $\pi:\tilde{M}\rightarrow M$.
\end{prop}

The above proposition will follow easily by the next 

\begin{lem}
\label{Mtool}
Let $A_1,\dots,A_N$ and $B_1,\dots,B_N$ be  finite sequences of open subset of $\tilde{M}$ such that
\begin{enumerate}
\item $A_i$ is a fundamental domain of $\pi\colon \tilde{M}\rightarrow M$ for each $i=1,\dots,N$;
\item $B_i$ is a relatively compact open subset with smooth boundary of $\tilde{M}$ such that $\overline{B_i}\subset A_i$, for each $i=1,\dots,N$;
\item $\{\pi(B_i)\}_{1,\dots,N}$ is an open cover of $M$.
\end{enumerate}
Given $\lambda_0 > 0$, there exists $\varepsilon_0(\lambda_0)>0$ such that for any open subset $U$ of  $M$ satisfying  
$$
\vol_g(U)<\varepsilon_0(\lambda_0),
$$
the following inequality holds
$$
\int_{\pi_1(M).B_i\cap \tilde{U}}|s|_{\tilde{\rho}}^2\dvol_{\tilde{g}}\leq \frac{1}{2N}\int_{\tilde{M}}|s|^2_{\tilde{\rho}}\dvol_{\tilde{g}},
$$ 
for any $s\in \im(E(\lambda_0))$ and $i=1,\dots,N$. 
\end{lem}

\begin{proof}
The proof of this lemma is similar to that of  \cite[Lemma 3.4]{BDET}. However since there are some minor variations we reproduce it for the benefit of the reader. 

Throughout the proof let us fix arbitrarily a pair $B_i\subset A_i$ that for simplicity we will denote with $B\subset A$. Let $s\in \im(E(\lambda_0))$ be arbitrarily fixed. Since $s\in \im(E(\lambda_0))$ we have that $s\in \mathcal{D}(\tilde{D}^n)$ for each positive integer $n$ and 
$$
\int_{\tilde{M}}|\tilde{D}^ns|^2_{\tilde{\rho}}\dvol_{\tilde{\rho}}\leq \lambda_0^{2n}\int_{\tilde{M}}|s|^2_{\tilde{\rho}}\dvol_{\tilde{g}}.
$$
Therefore, 
$$
\sum_{\gamma\in \pi_1(M)}\int_{\gamma A}(|\tilde{D}^ns|^2_{\tilde{\rho}}-\lambda_0^{2n}|s|^2_{\tilde{\rho}})\dvol_{\tilde{g}}\leq 0
$$ 
and thus there exists at least an element $\overline{\gamma}\in \pi_1(M)$ such that 
$$
\int_{\overline{\gamma} A}(|\tilde{D}^ns|^2_{\tilde{\rho}}-\lambda_0^{2n}|s|^2_{\tilde{\rho}})\dvol_{\tilde{g}}\leq 0.
$$ 
Let us fix a positive integer $\ell$ such that $2\ell d>m$. Next, define $\frak{S}(s)\subset \pi_1(M)$ as 
$$
\frak{S}(s)\colon =\left\{\gamma \in \pi_1(M)\colon  \int_{\gamma A}|\tilde{D}^{\ell}s|^2_{\tilde{\rho}}\dvol_{\tilde{g}}\leq 4N\lambda_0^{2\ell}\int_{\gamma A}|s|_{\tilde{\rho}}^2\dvol_{\tilde{g}}\right\}.
$$ 
We have 

\begin{multline*}
\int_{\tilde{M}}|s|^2_{\tilde{\rho}}\dvol_{\tilde{g}}
\geq \lambda_0^{-2\ell}\int_{\tilde{M}}|\tilde{D}^{\ell}s|^2_{\tilde{\rho}}\dvol_{\tilde{g}} \\
\geq \sum_{\gamma \notin \frak{S}(s)} \lambda_0^{-2\ell}\int_{\gamma A}|\tilde{D}^{\ell}s|^2_{\tilde{\rho}}\dvol_{\tilde{g}}\geq 
 \sum_{\gamma \notin \frak{S}(s)} 4N\int_{\gamma A}|s|^2_{\tilde{\rho}}\dvol_{\tilde{g}}.
\end{multline*}
Thus, we can deduce that for any   $U$ and $ \tilde{U}$ as above,

\begin{equation}
\label{notinS}
\sum_{\gamma \notin \frak{S}(s)} \int_{\gamma A\cap \tilde{U}}|s|^2_{\tilde{\rho}}\dvol_{\tilde{g}}\leq \frac{1}{4N}\int_{\tilde{M}}|s|^2_{\tilde{g}}\dvol_{\tilde{g}}.
\end{equation}
Let us now consider any $\gamma\in \frak{S}(s)$.
Since $\gamma^{-1}\colon (\tilde{M},\tilde{g})\rightarrow (\tilde{M},\tilde{g})$ is an isometry and $\gamma_E\colon \tilde{E}\rightarrow \tilde{E}$ is a fiberwise isometry we have

\begin{align}
\label{contro}
\int_A|\tilde{D}^{\ell}\gamma(s)|_{\tilde{\rho}}^2\dvol_{\tilde{g}}
&=\int_{ A}|\gamma(\tilde{D}^{\ell}s)|_{\tilde{\rho}}^2\dvol_{\tilde{g}}=\int_{ A}|\tilde{D}^{\ell}s|_{\tilde{\rho}}^2\circ \gamma^{-1}\dvol_{\tilde{g}}\\
\nonumber &=\int_{ \gamma A}|\tilde{D}^{\ell}s|_{\tilde{\rho}}^2\dvol_{\tilde{g}}\leq 4N\lambda_0^{2\ell}\int_{\gamma A}|s|_{\tilde{\rho}}^2\dvol_{\tilde{g}}\\
\nonumber &=4N\lambda_0^{2\ell}\int_A|\gamma(s)|_{\tilde{\rho}}^2\dvol_{\tilde{g}}.
\end{align}

Thanks to the elliptic estimates and the fact that $2d\ell>m$, see \textsl{e.g.} \cite[Lemma 1.1.17]{Les97}, we know that there exists a positive constant $C$ such that for any $x\in B$ and $\psi\in \mathcal{D}(\tilde{D}^{\ell})$ we have 
$$
|\psi(x)|_{\tilde{\rho}}^2\leq C\left(\|\psi\|^2_{L^2(A,\tilde{E}|_A,\tilde{g}|_A,\tilde{\rho}|_A)}+\|\tilde{D}^{\ell}\psi\|^2_{L^2(A,\tilde{E}|_A,\tilde{g}|_A,\tilde{\rho}|_A)}\right).
$$
Thus, for $\gamma\in \frak{S}(s)$, we get thanks to \eqref{contro} 
\begin{multline*}
|(\gamma(s))|_B|^2_{\tilde{\rho}}
\leq C\left(\|\gamma(s)\|^2_{L^2(A,\tilde{E}|_A,\tilde{g}|_A,\tilde{\rho}|_A)}+\|\tilde{D}^{\ell}(\gamma(s))\|^2_{L^2(A,\tilde{E}|_A,\tilde{g}|_A,\tilde{\rho}|_A)}\right) \\
\leq  C(1+4N\lambda_0^{2\ell})\|\gamma(s)\|^2_{L^2(A,\tilde{E}|_A,\tilde{g}|_A,\tilde{\rho}|_A)}
\end{multline*}
and therefore 
\begin{multline}
\label{Sobolevxz}
\int_{B\cap \tilde{U}}|\gamma(s)|^2_{\tilde{\rho}}\dvol_{\tilde{g}}
\leq \vol_{\tilde{g}}(B\cap \tilde{U})C(1+4N\lambda_0^{2\ell})\|\gamma(s)\|^2_{L^2(A,\tilde{E}|_A,\tilde{g}|_A,\tilde{\rho}|_A)} \\
=\vol_{\tilde{g}}(B\cap \tilde{U})C(1+4N\lambda_0^{2\ell})\int_A|\gamma(s)|^2_{\tilde{\rho}}\dvol_{\tilde{g}}.
\end{multline}
 Finally if we choose $\varepsilon_0(\lambda_0) <\frac{1}{4NC(1+4N\lambda_0^{2\ell})}$, we find 
 $$
 \vol_{\tilde{g}}(B\cap \tilde{U})C(1+4N\lambda_0^{2\ell})<\frac{1}{4N}.
 $$ 
 Then we have:
$$
\begin{aligned}
\int_{\pi_1(M).B\cap \tilde{U}}|s|^2_{\tilde{\rho}}\dvol_{\tilde{g}} 
& =\sum_{\gamma \in \frak{S}(s)}\int_{\gamma B\cap \tilde{U}}|s|^2_{\tilde{\rho}}\dvol_{\tilde{g}}+\sum_{\gamma \notin \frak{S}(s)}\int_{\gamma B\cap \tilde{U}}|s|^2_{\tilde{\rho}}\dvol_{\tilde{g}} \\
\textrm{(by \eqref{notinS})} & \leq  \sum_{\gamma \in \frak{S}(s)}\int_{\gamma B\cap \tilde{U}}|s|^2_{\tilde{\rho}}\dvol_{\tilde{g}}+\frac{1}{4N}\int_{\tilde{M}}|s|^2_{\tilde{\rho}}\dvol_{\tilde{g}} \\
& =\sum_{\gamma \in \frak{S}(s)}\int_{\gamma (B\cap \tilde{U})}|s|^2_{\tilde{\rho}}\dvol_{\tilde{g}}+\frac{1}{4N}\int_{\tilde{M}}|s|^2_{\tilde{\rho}}\dvol_{\tilde{g}} \\
& = \sum_{\gamma \in \frak{S}(s)}\int_{B\cap \tilde{U}}|\gamma(s)|^2_{\tilde{\rho}}\dvol_{\tilde{g}}+\frac{1}{4N}\int_{\tilde{M}}|s|^2_{\tilde{\rho}}\dvol_{\tilde{g}} \\
\textrm{(by \eqref{Sobolevxz})} & \leq \sum_{\gamma \in \frak{S}(s)}\frac{1}{4N}\int_{A}|\gamma(s)|^2_{\tilde{\rho}}\dvol_{\tilde{g}}+\frac{1}{4N}\int_{\tilde{M}}|s|^2_{\tilde{\rho}}\dvol_{\tilde{g}} \\
& =\sum_{\gamma \in \frak{S}(s)}\frac{1}{4N}\int_{\gamma A}|s|^2_{\tilde{\rho}}\dvol_{\tilde{g}}+\frac{1}{4N}\int_{\tilde{M}}|s|^2_{\tilde{\rho}}\dvol_{\tilde{g}} \\
& \leq \frac{1}{4N}\int_{\tilde{M}}|s|^2_{\tilde{\rho}}\dvol_{\tilde{g}}+\frac{1}{4N}\int_{\tilde{M}}|s|^2_{\tilde{\rho}}\dvol_{\tilde{g}} \\
& =\frac{1}{2N}\int_{\tilde{M}}|s|^2_{\tilde{\rho}}\dvol_{\tilde{g}}.
\end{aligned}
$$
If we repeat the above procedure for each pair $B_i\subset A_i$ and we choose $\varepsilon_0(\lambda_0)>0$ in such a way that 
$$
\vol_{\tilde{g}}(B_i\cap \tilde{U})C(1+4N\lambda_0^{2\ell})<\frac{1}{4N},\quad \mathrm{for}\ i=1,\dots,N,
$$ 
we can conclude that for any arbitrarily fixed $\lambda_0>0$ there exists $\varepsilon_0(\lambda_0)>0$ such that if $U$ is an open set in $M$ satisfying $\vol_g(U)<\varepsilon_{0}(\lambda_0)$ we then have
$$\int_{\pi_1(M).B_i\cap \tilde{U}}|s|_{\tilde{\rho}}^2\dvol_{\tilde{g}}\leq \frac{1}{2N}\int_{\tilde{M}}|s|^2_{\tilde{\rho}}\dvol_{\tilde{g}}$$
for any $s\in \im(E(\lambda_0))$ and $i=1,\dots,N$,  as desired.
\end{proof}

Endowed with Lemma \ref{Mtool} we can now prove Proposition \ref{Normcontrol}.

\begin{proof}[Proof of Proposition \ref{Normcontrol}]
We have
\begin{multline*}
\int_{\tilde{U}}|s|^2_{\tilde{\rho}}\dvol_{\tilde{g}}\leq \sum_{i=1}^N\int_{\pi_1(M).B_i\cap \tilde{U}}|s|^2_{\tilde{\rho}}\dvol_{\tilde{g}} \\ 
\textrm{(by Lemma \eqref{Mtool})}\ \leq  \sum_{i=1}^N\frac{1}{2N}\int_{\tilde{M}}|s|^2_{\tilde{\rho}}\dvol_{\tilde{g}}= \frac{1}{2}\int_{\tilde{M}}|s|^2_{\tilde{\rho}}\dvol_{\tilde{g}}.
\end{multline*}
Thus, 
$$
\int_{\tilde{U}}|s|^2_{\tilde{\rho}}\dvol_{\tilde{g}}\leq \frac{1}{2}\int_{\tilde{M}\setminus \tilde{U}}|s|^2_{\tilde{\rho}}\dvol_{\tilde{g}}+\frac{1}{2}\int_{\tilde{U}}|s|^2_{\tilde{\rho}}\dvol_{\tilde{g}}
$$ 
and so we reach the desired conclusion 
$$
\int_{\tilde{U}}|s|^2_{\tilde{\rho}}\dvol_{\tilde{g}}\leq \int_{\tilde{M}\setminus \tilde{U}}|s|^2_{\tilde{\rho}}\dvol_{\tilde{g}},
$$
and the proposition follows.
\end{proof}

We are finally in a good position to prove Theorem \ref{pspec}.

\begin{proof}[Proof of Theorem \ref{pspec}] 
Let $Z:=M\setminus A$. Let $\lambda_0 = 1,$ $\varepsilon_0(1)>0$ and $U\supset Z$  be as in Prop. \ref{Normcontrol}. By the fact that $L+\min(\sigma(\tilde{D}))>0$ over $A$ and $M\setminus U\subset A$, we know that there exists a constant $C>0$ such that  $L+\min(\sigma(\tilde{D}))\geq C$ over
$M\setminus U$. Thus, thanks to Prop. \ref{Normcontrol}, we get
$$
\begin{aligned}
\frac{C}{2}\int_{\tilde{M}}|s|^2_{\tilde{\rho}}\dvol_{\tilde{g}} & \leq C\int_{{\tilde{M}\setminus \tilde{U}}}|s|^2_{\tilde{\rho}}\dvol_{\tilde{g}} \\
& \leq \int_{{\tilde{M}\setminus \tilde{U}}}\left(\tilde{\rho}(Ls,s)+\min(\sigma(\tilde{D}))|s|^2_{\tilde{\rho}}\right)\dvol_{\tilde{g}}\\
&\leq \int_{{\tilde{M}}}\left(\tilde{\rho}(Ls,s)+\min(\sigma(\tilde{D}))|s|^2_{\tilde{\rho}}\right)\dvol_{\tilde{g}}\\
&\leq \int_{{\tilde{M}}}\left(\tilde{\rho}(\tilde{D}s,s)+\tilde{\rho}(Ls,s)\right)\dvol_{\tilde{g}}\\
&=\int_{\tilde{M}} \tilde{\rho}(\tilde{P}s,s)\dvol_{\tilde{g}}
\end{aligned}
$$ 
for each $s\in \im(E(1))$. Put it differently, we have just proved that there exists a positive constant $C$ such that
$$
\langle \tilde{P}s,s\rangle_{L^2(\tilde{M},\tilde{E},\tilde{g},\tilde{\rho})}\geq\frac C2\langle s,s\rangle_{L^2(\tilde{M},\tilde{E},\tilde{g},\tilde{\rho})}
$$ for each $s\in \im(E(1))$. It is now easy to conclude, see \textsl{e.g.} \cite[p. 28]{BDET}, that given  any $w\in \mathcal{D}(\tilde{P})$ we have  

$$ 
\langle\tilde{P}w,w\rangle_{L^2(\tilde{M},\tilde{E},\tilde{g},\tilde{\rho}))} \geq K\|w\|^2_{L^2(\tilde{M},\tilde{E},\tilde{g},\tilde{\rho}))},
$$
with $K:=\min\{C/2,1\}$ and this clearly implies that $0\notin \sigma(\tilde{P})$, as required.
\end{proof}

\section{Homologically nonsingular hyperbolic cohomology classes and first spectral consequences}

Let us now consider the Laplace--Beltrami operator $\tilde{\Delta}\colon C^{\infty}_c(\tilde{N})\rightarrow C^{\infty}_c(\tilde{N})$ and let $$\tilde{\Delta}\colon L^2(\tilde{N},\tilde{h})\rightarrow L^2(\tilde{N},\tilde{h})$$ be its unique closed extension. 
Denote with $\tilde{\lambda}_{0,h}$ the bottom of the spectrum of $\tilde{\Delta}$, that is $$\tilde{\lambda}_{0,h}:=\min(\sigma(\tilde{\Delta})).$$ We have the following

\begin{prop}
\label{C-In}
Let $N$ be a compact  manifold of dimension $n$ such that  
$$
V^{n}_{\mathrm{hyp}}(N)=H^{n}(N,\mathbb R).
$$ 
Then, for any arbitrarily fixed Riemannian metric $h$, we have $$\tilde{\lambda}_{0,h}>0.$$
\end{prop}

\begin{rem}
The hypothesis is fulfilled for instance whenever $\pi_1(N)$ is not amenable, cf. \cite{Sik01}. In particular the proposition holds true if $N$ is Kähler topologically hyperbolic by Proposition \ref{prop:amenable}.
\end{rem}

\begin{proof}
Since $V^{n}_{\mathrm{hyp}}(N)=H^{n}(N,\mathbb R)$, we know that there exists a smooth $(n-1)$-form, $\beta\in L^{\infty}\Omega^{n-1}(\tilde{N},\tilde{h})$, such that $d\beta=\dvol_{\tilde{h}}$, with $\tilde{h}=\pi^*h$. From the existence of such a form $\beta$ we obtain a linear isoperimetric inequality with isoperimetric constant given by $\|\beta\|_{L^{\infty}\Omega^1(\tilde{N},\tilde{h})}$. 

Indeed let $\Omega\subset \tilde{N}$ be a relatively compact open subset with $C^1$-boundary. We have 
$$
\begin{aligned}
\vol_{\tilde{h}}(\Omega)&= \int_{\Omega}\dvol_{\tilde{h}}=\int_{\Omega}d\beta\\
&= \int_{\partial \Omega} i^*\beta\leq \int_{\partial \Omega} |i^*\beta|_{\tilde{h}}\dvol_{i^*\tilde{h}}\\
&\leq \|\beta\|_{L^{\infty}\Omega^1(\tilde{M},\tilde{h})}\int_{\partial \Omega} \dvol_{i^*\tilde{h}},
\end{aligned}
$$ 
where $i^*\tilde{h}$ and $i^*\beta$ are the metric and the $(n-1)$-form on $\partial \Omega$ induced by the pull-back given by the inclusion $i\colon\partial \Omega\hookrightarrow \tilde{N}$. The thesis now follows from the Cheeger inequality, see \cite[Th. VI.1.2]{Chavel}.
\end{proof}

Using the above proposition we get a similar result for the Bochner Laplacian. More precisely:

\begin{cor}
\label{C-In2}
Let $(N,h)$ be a compact Riemannian manifold such that $V^{n}_{\mathrm{hyp}}(N)=H^{n}(N,\mathbb R)$. Let $(F,\tau)\rightarrow N$ be a Hermitian vector bundle and let $(\tilde{F},\tilde{\tau})\rightarrow \tilde{N}$ be the corresponding lift. Given an arbitrarily fixed connection $\nabla$ on $\tilde{F}$ compatible with the metric $\tilde{\tau}$ let $$\nabla^t\circ \nabla\colon L^2(\tilde{N},\tilde{F})\rightarrow L^2(\tilde{N},\tilde{F})$$ be the unique closed extension of the Bochner Laplacian acting on $C^{\infty}_c(\tilde{N},\tilde{F})$. Then $$\sigma(\nabla^t\circ \nabla)\subset [\tilde{\lambda}_{0,h},\infty).$$
\end{cor}

\begin{proof}
 Let $s\in C^{\infty}_c(\tilde{N},\tilde{F})$. The Kato inequality, see \textsl{e.g.} \cite[Prop.3.1]{Bei}, tells us that $|s|_{\tau}$ lies in the domain of $$d\colon L^2(\tilde{N},\tilde{h})\rightarrow L^2\Omega^1(\tilde{N},\tilde{h})$$ the unique $L^2$-closed extension of $d\colon C^{\infty}_c(\tilde{N})\rightarrow \Omega^1_c(\tilde{N})$ and 
 $$
 |d|s|_{\tilde{\tau}}|_{\tilde{h}}\leq |\nabla s|_{\tilde{h}\otimes \tilde{\tau}}.
 $$ 
 Taking the square and integrating over $N$, from the above inequality we can deduce that $$\langle \nabla s,\nabla s\rangle_{L^2(\tilde{N}, T^*\tilde{N}\otimes \tilde{F})}\geq \langle d(|s|_{\tilde{\tau}}),d(|s|_{\tilde{\tau}})\rangle_{L^2(\tilde{N},\tilde{h})}$$ for each $s\in  C^{\infty}_c(\tilde{N},\tilde{F})$. Now the conclusion follows immediately by Prop. \ref{C-In}.
\end{proof}

\section{Positive holomorphic vector bundles on K\"ahler topologically hyperbolic manifolds}

We are now in the position to collect some applications to K\"ahler topologically hyperbolic manifolds and positive holomorphic vector bundles. 
The main result of this section is the following.

\begin{thm}
\label{nsp}
Let $(M,h)$ be a K\"ahler topologically hyperbolic manifold. Let $(E,\tau)\rightarrow M$ be a Hermitian holomorphic vector bundle, Nakano positive over an open subset $A\subset M$ of full measure. Then  $$H^{m,0}_{\overline{\partial}_E}(M,E)\neq \{0\}$$ and consequently $$h^{m,0}_{\overline{\partial}_E}(M,E)=\chi(M,K_M\otimes E)>0.$$
\end{thm}

\begin{rem}\label{rem:proj}
By the positivity hypothesis on $E$, we have \textsl{a fortiori} that $\det E$ is semi-positive and strictly positive on an open subset (of full measure). Thus, it is big (and nef, indeed) by the solution of the Grauert--Riemenschneider conjecture by Siu and Demailly \cite{Siu84,Siu85,Dem85}. Therefore, by Moishezon's theorem, $M$ is projective algebraic since it is compact K\"ahler and carries a big line bundle.
\end{rem}

The proof of the above theorem relies on the next proposition:

\begin{prop}
\label{L2spec}
In the setting of Th. \ref{nsp} the following properties hold true:
\begin{enumerate}
\item  the value $0$ does not belong to the spectrum of $$\tilde{\Delta}_{\overline{\partial}_{\tilde{E}},m,q}\colon L^2\Omega^{m,q}(\tilde{M},\tilde{E})\rightarrow L^2\Omega^{m,q}(\tilde{M},\tilde{E})$$  for any $q\geq 1$;
\item the image of $$\tilde{\Delta}_{\overline{\partial}_{\tilde{E}},m,q}\colon L^2\Omega^{m,q}(\tilde{M},\tilde{E})\rightarrow L^2\Omega^{m,q}(\tilde{M},\tilde{E})$$ is closed for each $q=0,\dots,m$;
\item the value $0$ is an eigenvalue of $$\tilde{\Delta}_{\overline{\partial}_{\tilde{E}},m,0}\colon L^2\Omega^{m,0}(\tilde{M},\tilde{E})\rightarrow L^2\Omega^{m,0}(\tilde{M},\tilde{E}).$$
\end{enumerate}
\end{prop}

\begin{proof}
Let $\Theta(E)\in C^{\infty}(M,\Lambda^{1,1}(M)\otimes \mathrm{End}(E))$ be the curvature of the Chern connection of $(E,\tau)$ and let $\Lambda\colon\Lambda^{p,q}(M)\otimes E\rightarrow \Lambda^{p-1,q-1}(M)\otimes E$ be the fiberwise adjoint of $L\otimes \mathrm{Id}_E\colon \Lambda^{p-1,q-1}(M)\otimes E\rightarrow \Lambda^{p,q}(M)\otimes E$, with $\omega$ the K\"ahler form of $(M,h)$ and $$L\colon \Lambda^{p-1,q-1}(M)\rightarrow \Lambda^{p,q}(M),\quad L(\alpha):=\omega\wedge \alpha$$ the corresponding Lefschetz operator. According to the Akizuki--Nakano identity  we can decompose $\Delta_{\overline{\partial}_{E},p,q}$ as 
\begin{equation}
\label{Ak--Na}
\Delta_{\overline{\partial}_{E},p,q}=\Delta'_{E,p,q}+[i\Theta(E),\Lambda],
\end{equation}
 with $\Delta'_{E,p,q}\colon C^{\infty}(M,\Lambda^{p,q}(M)\otimes E)\rightarrow C^{\infty}(M,\Lambda^{p,q}(M)\otimes E)$ the non-negative $2$-nd order elliptic differential operator induced by the $(1,0)$-component of the Chern connection of $(E,\tau)\rightarrow M$, see \textsl{e.g.} \cite[\S 13.2]{BDIP}. Analogously on $(\tilde{M},\tilde{h})$ we have the following decomposition $$\Delta_{\overline{\partial}_{\tilde{E}},p,q}=\Delta'_{\tilde{E},p,q}+[i\Theta(\tilde{E}),\tilde{\Lambda}].$$ Since $(E,\tau)$ is Nakano-positive over $A$ the curvature term $[i\Theta(E),\Lambda]$ verifies  $[i\Theta(E),\Lambda]> 0$ on $\Lambda^{m,q}(A)\otimes E|_{A}$ for any $q\geq 1$. 
 
 Finally, since the operator $\Delta'_{\tilde{E},p,q}$ is non-negative and $M\setminus A$ has measure zero, we can conclude by Th. \ref{pspec} that $0$ does not belong to the spectrum of $\tilde{\Delta}_{\overline{\partial}_{\tilde{E}},m,q}\colon L^2\Omega^{m,q}(\tilde{M},\tilde{E})\rightarrow L^2\Omega^{m,q}(\tilde{M},\tilde{E})$ for any $q\geq 1$. The first point is thus proved. 
 
 We tackle now the second point. The fact that the image of 
 $$
 \tilde{\Delta}_{\overline{\partial}_{\tilde{E}},m,q}\colon L^2\Omega^{m,q}(\tilde{M},\tilde{E})\rightarrow L^2\Omega^{m,q}(\tilde{M},\tilde{E})
 $$ 
 is closed for each $q\geq 1$ follows immediately by the first point of this proposition. To show that also $\Delta_{\overline{\partial}_{\tilde{E}},m,0}$ has closed range we argue as follows.  Since $0\notin \sigma(\Delta_{\overline{\partial}_{\tilde{E}},m,1})$ we know that $\im(\Delta_{\overline{\partial}_{\tilde{E}},m,1})$ is closed in $L^2\Omega^{m,1}(\tilde{M},\tilde{E})$. Consequently 
 $$
 \overline{\partial}_{\tilde{E},m,0}\colon L^2\Omega^{m,0}(\tilde{M},\tilde{E})\rightarrow L^2\Omega^{m,1}(\tilde{M},\tilde{E})
 $$ 
 has closed range and this in turn implies that its adjoint 
 $$
 \overline{\partial}_{E,m,0}^*\colon L^2\Omega^{m,1}(\tilde{M},\tilde{E})\rightarrow L^2\Omega^{m,1}(\tilde{M},\tilde{E})
 $$ 
 has closed range, too. Finally since both $\overline{\partial}_{\tilde{E},m,0}$ and $\overline{\partial}_{\tilde{E},m,0}^*$ have closed range we can conclude that also $$\overline{\partial}_{\tilde{E},m,0}^*\circ \overline{\partial}_{\tilde{E},m,0}\colon L^2\Omega^{m,0}(\tilde{M},\tilde{E})\rightarrow L^2\Omega^{m,0}(\tilde{M},\tilde{E})$$ has closed range, that is $\im(\Delta_{\overline{\partial}_{\tilde{E}},m,0})$ is closed in $L^2\Omega^{m,0}(\tilde{M},\tilde{E})$ as required. 
 
To prove the third point we first need to recall the so-called Gromov--Vafa--Witten trick. Since the argument given in \cite{Gro91} applies \textsl{verbatim} to our setting, we provide only a brief account and  we refer to the aforementioned reference, as well as to \cite{Bal06} and \cite{Eys97} for details. 

Let us consider the operator 
$$
\sqrt{2}\overline{\eth}_{E,m}\colon L^2\Omega^{m,\bullet}(\tilde{M},\tilde{E})\rightarrow L^2\Omega^{m,\bullet}(\tilde{M},\tilde{E})
$$ 
that is the unique closed (and hence self-adjoint) extension of 
$$
\sqrt{2}(\overline{\partial}_{\tilde{E},m}+\overline{\partial}_{\tilde{E},m}^t)\colon \Omega_c^{m,\bullet}(\tilde{M}, \tilde{E})\rightarrow \Omega^{m,\bullet}_c(\tilde{M},\tilde{E}),
$$ 
with $\Omega_c^{m,\bullet}(\tilde{M},\tilde{E}):=\bigoplus_{q=0}^m\Omega^{m,q}_c(\tilde{M},\tilde{E})$ and  $\sqrt{2}(\overline{\partial}_{\tilde{E},m}+\overline{\partial}^t_{\tilde{E},m})$ the corresponding Dirac operator. Let $F$ be the trivial line bundle $\tilde{M}\times \mathbb{C}\rightarrow \tilde{M}$ endowed with the standard Hermitian metric and flat connection $\nabla_0$.

Given any $s>0$ let $\nabla^s$ be the connection on $F$ defined as $\nabla^s:=\nabla_0+is\eta$, with  $\eta\in \Omega^1(\tilde{M})\cap L^{\infty}\Omega^1(\tilde{M},\tilde{\omega})$ satisfying $d\eta=\tilde{\mu}$. Note that $\nabla^s$ is a metric connection for each $s>0$. 
Let 
\begin{multline*}
\sqrt{2}(\overline{\partial}_{\tilde{E},m}+\overline{\partial}_{\tilde{E},m}^t)\otimes \nabla^s\colon  \\ C_c^{\infty}(\tilde{M},\Lambda^{m,\bullet}(\tilde{M})\otimes \tilde{E}\otimes F)\rightarrow C_c^{\infty}(\tilde{M},\Lambda^{m,\bullet}(\tilde{M})\otimes\tilde{E}\otimes F)
\end{multline*}
be the first order differential operator obtained by twisting the Dirac operator $\sqrt{2}(\overline{\partial}_{\tilde{E},m}+\overline{\partial}_{\tilde{E},m}^t)$ with the connection $\nabla^s$, see \textsl{e.g.} \cite[p. 111]{Bal06}. Finally let us denote by 
\begin{equation}
\label{L2twist}
\overline{D}_{\tilde{E},m}^s\colon L^2(\tilde{M},\Lambda^{m,\bullet}(\tilde{M})\otimes \tilde{E}\otimes F)\rightarrow L^2(\tilde{M},\Lambda^{m,\bullet}(\tilde{M})\otimes \tilde{E}\otimes F)
\end{equation}
 the $L^2$-closure of $\sqrt{2}(\overline{\partial}_{\tilde{E},m}+\overline{\partial}_{\tilde{E},m}^t)\otimes \nabla^s$ and let $\overline{D}_{\tilde{E},m}^{s,+}$ (resp. $\overline{D}_p^{s,-}$) be the operator induced by \eqref{L2twist} with respect to the splitting given by $(m,\bullet)$-$E$ valued forms with even/odd anti-holomorphic degree. 
 
 Although $\overline{D}_{\tilde{E},m}^{s,+}$ is not equivariant with respect to the action of $\pi_1(M)$ it is possible for each fixed $s>0$ to construct a group $\Gamma_s$
such that the following properties hold true:
\begin{enumerate}
\item the group $\Gamma_s$ fits into a short exact sequence of groups:
$$
1\rightarrow U(1)\rightarrow \Gamma_s\stackrel{\rho_s}{\rightarrow} \pi_1(M)\rightarrow 1;
$$
\item the group $\Gamma_s$ acts on $F$ and its action preserves both the metric and the connection of $F$; 
\item  the group $\Gamma_s$ extends through the map $\rho_s$ to an action on 
$$
\Lambda^{m,\bullet}(\tilde{M})\otimes \tilde{E}\otimes F
$$ 
that preserves both the metric and the connection of $\Lambda^{m,\bullet}(\tilde{M})\otimes \tilde{E}\otimes F$;
\item we have $\tilde{p}\circ \gamma=\rho_s(\gamma)\circ \tilde{p}$ for each $\gamma\in \Gamma_s$,  with $\tilde{p}\colon \Lambda^{m,\bullet}(\tilde{M})\otimes \tilde{E}\otimes F\rightarrow \tilde{M}$ denoting the vector bundle projection;
\item  the group $\Gamma_s$ acts on $C^{\infty}(\tilde{M},\Lambda^{m,\bullet}(\tilde{M})\otimes \tilde{E}\otimes F)$ by $$\gamma s:=\gamma\circ s\circ \rho_s(\gamma^{-1})$$ with $s\in C^{\infty}(\tilde{M},\Lambda^{m,\bullet}(\tilde{M})\otimes \tilde{E}\otimes F)$;
\item  $\overline{D}_{\tilde{E},m}^{s,+/-}$ is equivariant w.r.t. the action of $\Gamma_s$ on $C^{\infty}(\tilde{M},\Lambda^{m,\bullet}(\tilde{M})\otimes \tilde{E}\otimes F)$.
\end{enumerate}
Let now 
$$
\pi_{K_s}^{+/-}\colon L^2(\tilde{M},\Lambda^{m,\bullet}(\tilde{M})\otimes \tilde{E}\otimes F)\rightarrow L^2(\tilde{M},\Lambda^{m,\bullet}(\tilde{M})\otimes \tilde{E}\otimes F)
$$ 
be the orthogonal projection on $\ker(\overline{D}_{\tilde{E},m}^{s,+/-})$.  According to \cite[Prop. 2.4]{Atiyah} we know that $\pi_{K_s}^{+/-}$ is an integral operator whose kernel, denoted here with $E_{K_s}^{+/-}$, is smooth. Now, it is not difficult to verify that the function $\mathrm{tr}(E_{K_s}^{+/-}(x,x))\colon \tilde{M}\rightarrow \mathbb{R}$ is $\pi_1(M)$-invariant. 
The  $L^2_{\Gamma_s}$-index of $\overline{D}_{\tilde{E},m}^{s,+/-}$ with respect to $\Gamma_s$ is thus  defined as $$L^2_{\Gamma_s}-\text{ind}(\overline{D}_{\tilde{E},m}^{s,+}):=\int_{U}\mathrm{tr}(E_{K_s}^+(x,x))\dvol_{\tilde{g}}-\int_{U}\mathrm{tr}(E_{K_s}^-(x,x))\dvol_{\tilde{g}}$$
with $U$ an arbitrarily fixed fundamental domain of $\pi:\tilde{M}\rightarrow M$. Arguing as in \cite[\S 7.2.2]{Eys97}, we can show now that the above $L^2_{\Gamma_s}$-index can be computed by using the heat operator associated to $\overline{\Delta}_{\tilde{E},m}^{s,+}:=\overline{D}_{\tilde{E},m}^{s,-}\circ\overline{D}_{\tilde{E},m}^{s,+}$ and $\overline{\Delta}_{\tilde{E},m}^{s,-}:=\overline{D}_{\tilde{E},m}^{s,+}\circ\overline{D}_{\tilde{E},m}^{s,-}$. More precisely: $$L^2_{\Gamma_s}-\text{ind}(\overline{D}_{\tilde{E},m}^{s,+}):=\int_{U}\mathrm{tr}(H_{K_s}^+(x,x,t))\dvol_{\tilde{g}}-\int_{U}\mathrm{tr}(H_{K_s}^-(x,x,t))\dvol_{\tilde{g}}$$
with $\mathrm{tr}(H_{K_s}^{+/-}(x,x,t))$ the pointwise trace of the kernel of the heat operator 
$$
e^{-t\overline{\Delta}_{\tilde{E},m}^{s,+/-}}\colon L^2(\tilde{M},\Lambda^{m,+/-}(\tilde{M})\otimes \tilde{E}\otimes F)\rightarrow L^2(\tilde{M},\Lambda^{m,+/-}(\tilde{M})\otimes \tilde{E}\otimes F)
$$ 
respectively. 

Finally, by applying the local index theorem for twisted spin-c Dirac operators, see \textsl{e.g.} \cite[Prop.13.1]{Duistermaat}, we get the desired formula:
\begin{equation}
\label{L2formula}
L^2_{\Gamma_s}-\text{ind}(\overline{D}_{\tilde{E},m}^{s,+})=\int_M\text{Todd}(M)\wedge\text{ch}(\Lambda^{m,0}(M))\wedge \text{ch}(E)\wedge \text{ch}(F).
\end{equation}
Note that $\text{ch}(F)=\text{exp}(-s\mu/2\pi)$.

By the fact that $\int_M\mu^m\neq 0$ we obtain that \eqref{L2formula} is a polynomial function in $s$ with non trivial leading coefficient and hence it has only isolated zeros, see \cite[pg. 281]{Gro91}. In particular for all but a discrete subset of real values $s$ we have $\ker(\overline{D}_{\tilde{E},m}^s)\neq \{0\}$.  We can therefore find an $\epsilon>0$ such that $\ker(\overline{D}_{\tilde{E},m}^s)\neq \{0\}$ for each $s\in (0,\epsilon)$ and thus, thanks to \cite[Prop.7.1.2]{Eys97}, we can conclude that $0\in \sigma(\overline{\eth}_{\tilde{E},m})$, that is $0$ lies in the spectrum of $\overline{\eth}_{\tilde{E},m}$. 

Since 
\begin{multline}
\label{directsum}
\overline{\eth}_{\tilde{E},m}^2\colon L^2\Omega^{m,\bullet}(\tilde{M},\tilde{E})\rightarrow L^2\Omega^{m,\bullet}(\tilde{M},\tilde{E})\\=\bigoplus_{q=0}^m\Delta_{\overline{\partial}_{\tilde{E}},m,q}\colon L^2\Omega^{m,\bullet}(\tilde{M},\tilde{E})\rightarrow L^2\Omega^{m,\bullet}(\tilde{M},\tilde{E})
\end{multline}
 we know that $0\in \sigma(\Delta_{\overline{\partial}_{\tilde{E}},m,q})$ for some $q=0,\dots,m$. We can thus deduce from the first point of this proposition that $0\notin  \sigma(\Delta_{\overline{\partial}_{\tilde{E}},m,q})$ for each $q=1,\dots,m$. Clearly this implies that $0\in \sigma(\Delta_{\overline{\partial}_{\tilde{E}},m,0})$.  
 
 Finally, since we know by the second point that $\Delta_{\overline{\partial}_{\tilde{E}},m,0}$ has closed range, we can conclude that $0$ is actually an eigenvalue of $\Delta_{\overline{\partial}_{\tilde{E}},m,0}$, that is 
 $$
 \ker(\Delta_{\overline{\partial}_{\tilde{E}},m,0})\neq \{0\}.
 $$ 
\end{proof}

\begin{proof}[Proof of Th. \ref{nsp}]
According to Atiyah's $L^2$-index theorem, see \cite{Atiyah}, we know that
$$\chi(M,K_M\otimes E)=\sum_{q=0}^mh^{m,q}_{(2),\overline{\partial}_{E}}(M,E).$$
Therefore, by Prop. \ref{L2spec}, we know that $$\chi(M,K_M\otimes E)=h^{m,0}_{(2),\overline{\partial}_{E}}(M,E)>0.$$ On the other hand, since $(E,\tau)\rightarrow M$ is Nakano positive over $A$, we know that $$H^{m,q}_{\overline{\partial}_E}(M,E)=\{0\}$$ for each $q\geq 1$.
We can thus conclude that $$H^{m,0}_{\overline{\partial}_{E}}(M,E)\neq \{0\}$$ as required.
\end{proof}

We collect now some corollaries that follow from the above results.

\begin{cor}
\label{L2holsec}
In the setting of Th. \ref{nsp}, the following properties hold true:
\begin{enumerate}
\item The $L^2$-Hodge numbers of $E\rightarrow M$ satisfy
$$h^{m,q}_{(2),\overline{\partial}_E}(M,E)=0$$ for each $q\geq 1$ and
$$h^{m,0}_{(2),\overline{\partial}_E}(M,E)>0.$$ Moreover $h^{m,0}_{(2),\overline{\partial}_E}(M,E)$ is an integer.
\item The space of $L^2$-holomorphic sections of $\tilde{K}_M\otimes \tilde{E}\rightarrow \tilde{M}$ is infinite dimensional (in the usual sense).
\end{enumerate}
\end{cor}

\begin{proof}
The first point is an immediate consequence of Prop. \ref{L2spec}. Concerning the second point we know that $$\tilde{\Delta}_{\overline{\partial}_{\tilde{E}},m,0}\colon L^2\Omega^{m,q}(\tilde{M},\tilde{E})\rightarrow L^2\Omega^{m,0}(\tilde{M},\tilde{E})$$ has non-trivial kernel. Since it is also invariant through the action of $\pi_1(M)$, we can conclude that it is infinite dimensional (in the usual sense), see \cite[Lemma 15.10]{Roe98}. Finally, given that the kernel of 
$$
\tilde{\Delta}_{\overline{\partial}_{\tilde{E}},m,0}\colon L^2\Omega^{m,q}(\tilde{M},\tilde{E})\rightarrow L^2\Omega^{m,0}(\tilde{M},\tilde{E})
$$ 
equals the kernel of 
$$
\overline{\partial}_{\tilde{E},m,0}\colon L^2\Omega^{m,0}(\tilde{M},\tilde{E})\rightarrow L^2\Omega^{m,1}(\tilde{M},\tilde{E}),
$$ 
we can conclude that the space of $L^2$-holomorphic sections of $\tilde{K}_M\otimes \tilde{E}\rightarrow \tilde{M}$ is infinite dimensional.
\end{proof}

\begin{cor}
\label{gsp}
Let $(M,h)$ be a K\"ahler topologically hyperbolic manifold. Let $(E,\tau)\rightarrow M$ be a rank $r\ge 2$ Hermitian holomorphic vector bundle, Griffiths positive over an open subset $A\subset M$ of full measure. Then  for each positive integer $\ell$ we have $$
H^{m,0}_{\overline{\partial}_E}\bigl(M,E^{\otimes\ell}\otimes (\det E)^{r\ell}\bigr)\neq \{0\}
$$ 
and 
$$\chi\bigl(M,K_M\otimes E^{\otimes\ell}\otimes (\det E)^{r\ell}\bigr)=h^{m,0}_{\overline{\partial}_{E}}\bigl(M, E^{\otimes\ell}\otimes (\det E)^{r\ell}\bigr)>0.
$$
\end{cor}

\begin{proof}
According to \cite{Dema-Skoda} the holomorphic vector bundle $E\otimes \det E$ endowed with the Hermitian metric induced by $h$ is Nakano positive over $A$. Moreover, the tensor product of Griffiths positive vector bundles is still Griffiths positive, and $\det E^{\otimes\ell}\simeq(\det E)^{r\ell}$. Now the conclusion follows immediately by Th. \ref{nsp}.
\end{proof}

\begin{rem}
More generally, one can look at irreducible $\operatorname{GL}(E)$-represen\-tations of $E^{\otimes\ell}$, in the spirit of \cite{Dem88a,Dem88p,Man97}. For a non-increasing weight $\underline a=(a_1,\dots,a_r)\in \mathbb N^r$, let $\Gamma^{\underline a} E$ be the irreducible tensor power representation of $\operatorname{GL}(E)$ of highest weight $\underline a$. We always have a canonical $\operatorname{GL}(E)$-isomorphism
$$
E^{\otimes\ell}\simeq\bigoplus_{\substack{a_1+\cdots+a_r=\ell \\ a_1\ge\cdots\ge a_r\ge 0}}\mu(\underline a,\ell)\,\Gamma^{\underline a} E,
$$
where $\mu(\underline a,\ell) > 0$ is the multiplicity of the isotypical factor $\Gamma^{\underline a} E$ in $E^{\otimes \ell}$. Since each $\Gamma^{\underline a} E$ is a quotient of $E^{\otimes\ell}$, Griffiths positivity propagates and we also have an analogous non-vanishing for $H^{m,0}_{\overline{\partial}_E}\bigl(M,\Gamma^{\underline a} E\otimes (\det E)^{\operatorname{rk}(\Gamma^{\underline a} E)}\bigr)$.  Morevoer, there is no need to twist by $\det E$ if $\operatorname{rk}\Gamma^{\underline a} E=1$, see the comment right below this remark. Thus, we have the more precise information that not only $H^0\bigl(M,K_M\otimes E^{\otimes\ell}\otimes (\det E)^{r\ell}\bigr)$ does not vanish but each of its direct summand 
$$
\begin{aligned}
& H^0\bigl(M,K_M\otimes \Gamma^{\underline a} E\otimes (\det E)^{\operatorname{rk}\Gamma^{\underline a} E)}\bigr)\ne\{0\},\quad \operatorname{rk}\Gamma^{\underline a} E\ge 2, \\
& H^0\bigl(M,K_M\otimes \Gamma^{\underline a} E\bigr)\ne\{0\},\quad \operatorname{rk}\Gamma^{\underline a} E= 1
\end{aligned}
$$ 
as well.
\end{rem}

We look now more closely at the case $r=1$, where there is no need of twisting by $\det E$ since Griffiths and Nakano positivity do coincide in this case.

\begin{cor}\label{line}
In the setting of Th. \ref{nsp}, assume additionally that $E$ is a line bundle. Then, for each positive integer $p$, we have 
$$
H^{m,0}_{\overline{\partial}_{E}}(M,E^p)\neq \{0\}.
$$ 
Moreover, $K_M\otimes E^p$ is a big line bundle for each positive integer $p$.
\end{cor}

\begin{proof}
For each positive integer $p$, the line bundle $E^p$ is still positive over $A$. Therefore by Theorem \ref{nsp} we know that $H^{m,0}_{\overline{\partial}_{E}}(M,E^p)\neq \{0\}.$ 

Now, $E$ is nef since it is semi-positive and big by the solution of the Grauert--Riemenschneider conjecture by Siu and Demailly \cite{Siu84,Siu85,Dem85}, being positive on an open set of full measure. Since $M$ is projective (cf. Remark \ref{rem:proj}), we have that $K_M$ is pseudoeffective by Corollary \ref{cor:pseff} and therefore $K_M\otimes E^p$ is big as a tensor product of a pseudoeffective and a big line bundle.
\end{proof}

\begin{rem}\label{rem:aikconj}
The above corollary is a confirmation of the Kawamata--Ambro--Ionescu effective non vanishing conjecture in this (very) special case. 

Indeed, (one of the form of) this conjecture states that on a projective manifold $X$, for any given big and nef line bundle $L\to X$ such that $K_X+L$ is nef one should have $H^0(X,K_X+L)\ne 0$. We have here somehow a weaker positivity on the adjoint bundle $K_X+L$ which is big (and hence pseudoeffective, but nefness is stronger than pseudoeffectivity) and not necessarily nef, but stronger positivity assumption on $L$ as well as restrictions on the topology of $X$.  

Beside the case of curves and surfaces, the only case where the conjecture is fully known is for threefolds \cite{Hor12}. On the other hand, the conjecture is proved in all dimensions in \cite[Théorème 1]{Eys99} (see also \cite{BH10,Cha07,Tak99}) under the topological hypothesis of large fundamental group in the sense of Koll\'ar and just bigness for $L$. It is thus somehow difficult to compare our situation with that of \cite{Eys99}. 
\end{rem}

Here is a special case about the canonical bundle under the stronger hypothesis of weak Kähler hyperbolicity.

\begin{cor}
Let $M$ be a weakly K\"ahler hyperbolic manifold without rational curves. Then 
\begin{equation}
\label{wkhcan}
H^{m,0}_{\overline{\partial}_{E}}(M,K_{M}^p)\neq \{0\}
\end{equation}
for each positive integer $p$. 
In particular, \eqref{wkhcan} holds true whenever $M$ is K\"ahler hyperbolic.
\end{cor}

\begin{proof} According to \cite[Th. 4.1]{BDET} we know that $K_M$ is big and thus, since we assumed additionally that $M$ has no rational curves, we can conclude that $K_M$ is ample by standard arguments in birational geometry, see \textsl{e.g.} \cite[Exercise 8 on p. 219]{Deb01}. 
If $M$ is K\"ahler hyperbolic then the absence of rational curves is automatic. 
The conclusion now follows in both cases by Cor. \ref{line}.
\end{proof}

\section{Curvature bounds on K\"ahler topologically hyperbolic manifolds}\label{sec:curv}

Th. \ref{nsp} can also be used to provide upper estimates of the negative part of the curvature of certain Hermitian homolorphic vector bundles. This is the contents of the next corollaries.

\begin{cor}
\label{up-es}
Let $(M,h)$ and $(E,\tau)\rightarrow M$ be as in Th. \ref{nsp}. 
Let us denote with $\mu_1:M\rightarrow \mathbb{R}$ the continuous function given by the lowest eigenvalues of $[i\Theta(E), \Lambda]$ acting on $(E,\tau)$ and let $c>1$ be the refined Kato constant of $K_M\otimes E$ w.r.t. $\overline{\partial}_{E,m,0}$. Then the following inequality holds true:
\begin{equation}
\label{min}
\min_M\left(\mathrm{scal}_h+2\mu_1\right)\leq -2c\tilde{\lambda}_{0,h},
\end{equation}
 and the equality occurs if and only if 
\begin{equation}
\label{min2}
\mathrm{scal}_h+2\mu_1\equiv -2c\tilde{\lambda}_{0,h}.
\end{equation}
\end{cor}

\begin{proof}
We argue by contradiction and thus we assume that $\mathrm{scal}_h+2\mu_1+2c\tilde{\lambda}_{0,h}\geq 0$ on $M$ and  $\mathrm{scal}_h(x)+2\mu_1(x)+2c\tilde{\lambda}_{0,h}>0$ for some $x\in M$. Let us consider the operator $$\tilde{\Delta}_{\overline{\partial}_E,m,0}\colon L^2\Omega^{m,0}(\tilde{M},\tilde{E})\rightarrow L^2\Omega^{m,0}(\tilde{M},\tilde{E}).$$ According to the Akizuki--Nakano identity with respect to the Chern connection of $K_{\tilde{M}}\otimes \tilde{E}$ we have 
$$
\begin{aligned}
\tilde{\Delta}_{\overline{\partial}_{\tilde{E}},m,0}&=\Delta'_{K_{\tilde{M}}\otimes\tilde{E}}+[i\Theta(K_{\tilde{M}}\otimes\tilde{E}),\tilde{\Lambda}]\\
&= \Delta'_{K_{\tilde{M}}\otimes\tilde{E}}+[i\Theta(K_{\tilde{M}}),\tilde{\Lambda}]+[i\Theta(\tilde{E}),\tilde{\Lambda}]\\
&=\Delta'_{K_{\tilde{M}}\otimes\tilde{E}}+\frac{\mathrm{scal}_{\tilde{h}}}{2}+[i\Theta(\tilde{E}),\tilde{\Lambda}],
\end{aligned}
$$
with $\Delta'_{K_{\tilde{M}}\otimes\tilde{E}}$ the second order elliptic operator acting on $C_c^{\infty}(\tilde{M},K_{\tilde{M}}\otimes \tilde{E})$ induced by the $(1,0)$ component of the Chern connection of $K_{\tilde{M}}\otimes \tilde{E}$. Let $\eta \in \ker(\tilde{\Delta}_{\overline{\partial}_{\tilde{E}},m,0})$. By denoting with $\tilde{\nabla}$ the Chern connection of $K_{\tilde{M}}\otimes \tilde{E}$ and keeping in mind the refined Kato inequality \cite{Ref-Kato} we have 
$$
\begin{aligned}
0&=\langle 2\tilde{\Delta}_{\overline{\partial}_{\tilde{E}},m,0}\eta,\eta\rangle_{L^2\Omega^{m,0}(\tilde{M},\tilde{E})}\\
&=\int_{\tilde{M}}\left(|\tilde{\nabla} \eta|^2_{\tilde{h}\otimes \tilde{\tau}}+\frac{\mathrm{scal}_{\tilde{h}}}{2}|\eta|^2_{\tilde{h}\otimes \tilde{\tau}}+\langle     [i\Theta(\tilde{E}),\tilde{\Lambda}]\eta,\eta\rangle_{\tilde{h}\otimes \tilde{\tau}}\right)\dvol_{\tilde{h}}\\
&\geq \int_{\tilde{M}}\left(c|d|\eta|_{\tilde{h}\otimes \tilde{\tau}}|^2_{\tilde{h}}+\frac{\mathrm{scal}_{\tilde{h}}}{2}|\eta|^2_{\tilde{h}\otimes \tilde{\tau}}+\mu_1|\eta|^2_{\tilde{h}\otimes \tilde{\tau}}\right)\dvol_{\tilde{h}}\\
(\mathrm{by\ Prop.}\ \ref{C-In}) &\geq \int_{\tilde{M}}\left(\left(c\tilde{\lambda}_{0,h}+\frac{\mathrm{scal}_{\tilde{h}}}{2}+\tilde{\mu}_1\right)|\eta|^2_{\tilde{h}\otimes \tilde{\tau}}\right)\dvol_{\tilde{h}}\\
& \geq 0,
\end{aligned}
$$
where $\tilde{\mu}_1=\mu_1\circ\pi$ denotes the smallest eigenvalues of $[i\Theta(\tilde{E}), \tilde{\Lambda}]$ acting on $(\tilde{E},\tilde{\tau})$.
Since the function $c\tilde{\lambda}_{0,h}+\frac{\mathrm{scal}_{\tilde{h}}}{2}+\tilde{\mu}_1$ is non-negative and positive somewhere on $\tilde{M}$ and $\eta$ is holomorphic we can conclude that $\eta$ vanishes identically. Thus $\ker(\tilde{\Delta}_{\overline{\partial}_{\tilde{E}},m,0})=\{0\}$. 

On the other hand by Th. \ref{nsp} we know that $\ker(\tilde{\Delta}_{\overline{\partial}_{\tilde{E}},m,0})\neq \{0\}$. We can therefore conclude that \eqref{min}-\eqref{min2} hold true as we have reached the desired contradiction.
\end{proof}

\begin{rem}
Note that \eqref{min}-\eqref{min2} are stronger than what one could expect by only knowing that $H^{m,0}_{\overline{\partial}_E,m,0}(M,E)\neq \{0\}$. Indeed this latter inequality implies only that $\min_M\left(\mathrm{scal}_h+2\mu_1\right)\leq 0$ whereas in our framework the presence of a homologically non-singular hyperbolic cohomology class of degree two allows us to estimates $\min_M\left(\mathrm{scal}_h+2\mu_1\right)$ with the negative constant $-2c\tilde{\lambda}_{0,h}$.
\end{rem}

We continue with the next

\begin{cor}
In the setting of Th. \ref{nsp}, for each $p\in \{1,\dots,m-1\}$ there exists at least one $q\in\{0,\dots,m\}$ such that, for any open subset $W\subset M$ with $\vol_h(W)=\vol_h(M)$, the curvature term $$[i\Theta(E),\Lambda]$$ is not positive definite on $\Lambda^{p,q}(W)\otimes E|_W$.
\end{cor}

\begin{proof}
The proof is carried out by contradiction. Assume that the statement does not hold true. Then there exists an integer $p\in \{1,\dots,m-1\}$ such that for each $q\in \{0,\dots,m\}$ there exists an open subset $W_q\subset M$ with $\vol_h(W_q)=\vol_h(M)$ such that the curvature term $[i\Theta(E),\Lambda]$ is positive definite on $\Lambda^{p,q}(W_q)\otimes E|_{W_q}$. Let $W=W_1\cap \dots\cap W_m$. Then $W$ is open, still verifies $\vol_h(W)=\vol_h(M)$ and now $[i\Theta(E),\Lambda]$ is positive definite on $\Lambda^{p,q}(W_q)\otimes E|_{W_q}$ for each $q=0,\dots,m$. According to the Akizuki--Nakano identity \eqref{Ak--Na} and Th. \ref{pspec} we can conclude that $0$ does not belong to the spectrum of $$\tilde{\Delta}_{\overline{\partial}_{\tilde{E}},p,q}\colon L^2\Omega^{p,q}(\tilde{M},\tilde{E})\rightarrow L^2\Omega^{p,q}(\tilde{M},\tilde{E})$$ for each $q=0,\dots,m$. Let us consider the operator $$\overline{\eth}_{E,p}\colon L^2\Omega^{p,\bullet}(\tilde{M},\tilde{E})\rightarrow L^2\Omega^{p,\bullet}(\tilde{M},\tilde{E})$$ that is  the unique closed (and hence self-adjoint) extension of 
$$
\overline{\partial}_{\tilde{E},p}+\overline{\partial}_{\tilde{E},p}^t\colon\Omega_c^{p,\bullet}(\tilde{M}, \tilde{E})\rightarrow \Omega^{p,\bullet}_c(\tilde{M},\tilde{E}),
$$ 
with $\Omega_c^{p,\bullet}(\tilde{M},\tilde{E}):=\bigoplus_{q=0}^m\Omega^{p,q}_c(\tilde{M},\tilde{E})$ and  $\overline{\partial}_{\tilde{E},p}+\overline{\partial}^t_{\tilde{E},p}$ the corresponding Dirac operator. Since 
\begin{multline}
\overline{\eth}_{\tilde{E},p}^2\colon L^2\Omega^{p,\bullet}(\tilde{M},\tilde{E})\rightarrow L^2\Omega^{p,\bullet}(\tilde{M},\tilde{E})\\=\bigoplus_{q=0}^m\Delta_{\overline{\partial}_{\tilde{E}},p,q}\colon L^2\Omega^{p,\bullet}(\tilde{M},\tilde{E})\rightarrow L^2\Omega^{p,\bullet}(\tilde{M},\tilde{E}),
\end{multline}
 we can conclude that  $0\notin \sigma(\overline{\eth}_{\tilde{E},p})$. On the other hand, arguing again as in the proof of Prop. \ref{L2spec}, let us consider the operator  
 $$
 \sqrt 2(\overline{\partial}_{\tilde{E},p}+\overline{\partial}_{\tilde{E},p}^t)\otimes \nabla^s\colon C_c^{\infty}(\tilde{M},\Lambda^{p,\bullet}(\tilde{M})\otimes \tilde{E}\otimes \mathbb{C})\rightarrow C_c^{\infty}(\tilde{M},\Lambda^{p,\bullet}(\tilde{M})\otimes\tilde{E}\otimes \mathbb{C})
 $$ 
 which is  obtained by twisting the Dirac operator $\sqrt 2(\overline{\partial}_{\tilde{E},p}+\overline{\partial}_{\tilde{E},p}^t)$ with the connection $\nabla^s$. Let us denote with 
\begin{equation}
\label{L2twistx}
\overline{D}_{\tilde{E},p}^s\colon L^2(\tilde{M},\Lambda^{p,\bullet}(\tilde{M})\otimes \tilde{E}\otimes\mathbb{C})\rightarrow L^2(\tilde{M},\Lambda^{p,\bullet}(\tilde{M})\otimes \tilde{E}\otimes \mathbb{C})
\end{equation}
 the $L^2$-closure of $\sqrt 2(\overline{\partial}_{\tilde{E},p}+\overline{\partial}_{\tilde{E},p}^t)\otimes \nabla^s$. Then the same argument recalled in Prop. \ref{L2spec} shows that $0$ lies in the spectrum of $\overline{D}_{\tilde{E},p}^s$ and consequently, by \cite[Prop. 7.1.2]{Eys97}, we can conclude that $0\in \sigma(\overline{\eth}_{\tilde{E},p})$. The proof  is thus complete as we reached the desired contradiction. 
\end{proof}

The same ideas that we have used so far can also be applied to estimate from above the negative part of the Ricci curvature on a K\"ahler topologically hyperbolic manifold. This is the content of the next results.\\

 Let $\mathrm{Ric}_h$ be the Ricci tensor of the metric $h$ and  let $\mathrm{ric}_h$ be the unique endomorphism of $TM$ such that $\mathrm{Ric}_h(\cdot,\cdot)=h(\mathrm{ric}_h \cdot, \cdot)$. Since $\mathrm{ric}_h$ commutes with $J$, the complex structure of $M$, each eigenspace of $\mathrm{ric}_h$ is preserved by $J$ and thus has even dimension. There exist therefore $m$ real-valued continuous functions $r_j:M\rightarrow \mathbb{R}$, $j=1,\dots,m$, such that 
$$
r_1\leq r_2\leq\dots\leq r_{m}
$$
 and for each $p\in M$ 
\begin{equation}
\label{eigenricci}
\{r_1(p),r_1(p),r_2(p),r_2(p),\dots,r_{m}(p),r_{m}(p)\}
\end{equation}
 is the set of eigenvalues of $\mathrm{ric}_{h,p}:T_pM\rightarrow T_pM$. Since $\mathrm{ric}_h$ commutes with $J$ his $\mathbb{C}$-linear extension to $TM\otimes \mathbb{C}$ preserves both $T^{1,0}M$ and $T^{0,1}M$. Let us denote with $\mathrm{ric}^{1,0}$ the restriction of $\mathrm{ric}_h$ to $T^{1,0}M$. Then for each $p\in M$ the $m$ eigenvalues of 
\begin{equation}
\label{ric10}
\mathrm{ric}_p^{1,0}\colon T_p^{1,0}M\rightarrow T_p^{1,0}M
\end{equation}
 are given by $$\{r_1(p),r_2(p),\dots,r_{m-1}(p),r_{m}(p)\}.$$ Finally, with a little abuse of notation, we still denote with $\mathrm{ric}^{1,0}$ the endomorphism of $\Lambda^{1,0}(M)$ induced by $\mathrm{ric}^{1,0}\colon T^{1,0}M\rightarrow T^{1,0}M$. We have all the ingredients for the next

\begin{prop}
\label{curvature1}
Let $M$ be a K\"ahler topologically hyperbolic manifold. Then, for any arbitrarily fixed K\"ahler metric $h$, there exists $p\in \{1,\dots,m\}$ such that the inequality $$r_1+r_2+\cdots+r_p\leq -\tilde{\lambda}_{0,h}$$ holds true over a closed subset of $M$ of positive measure. 
\end{prop}

\begin{proof}
We proceed by contradiction and so we assume that the above conclusion  does not hold. Let us denote with $\mathrm{ric}^{p,0}\colon\Lambda^{p,0}(M)\rightarrow \Lambda^{p,0}(M)$ the vector bundle endomorphism obtained by $\mathrm{ric}^{1,0}\colon\Lambda^{1,0}(M)\rightarrow \Lambda^{1,0}(M)$ extended as a derivation. Then for each $p\in \{1,\dots.,m\}$ there exists an open set $U_p\subset M$ such that $\vol_h(U_p)=\vol_h(M)$ and $\mathrm{ric}^{p,0}>-\tilde{\lambda}_{0,h}$ on $U_q$. Let $U:=U_1\cap\cdots\cap U_m$. Obviously we have $\vol_h(U)=\vol_h(M)$ and $\mathrm{ric}^{p,0}>-\tilde{\lambda}_{0,h}$ on $U$ for each $p=1,\dots,m$. Since the Weitzenb\"ock formula reads as 
$$
\Delta_{\overline{\partial},p,0}=\frac{1}{2}\left(\nabla^t\circ \nabla +\mathrm{ric}^{p,0}\right),
$$ 
we can apply Th. \ref{pspec} and Cor. \ref{C-In2} to conclude that $0$ does not lie in $\sigma(\tilde{\Delta}_{\overline{\partial},p,0})$, the spectrum of $$\tilde{\Delta}_{\overline{\partial},p,0}\colon L^2\Omega^{p,0}(\tilde{M},\tilde{h})\rightarrow L^2\Omega^{p,0}(\tilde{M},\tilde{h})$$ for each $p=1,\dots,m$. Furthermore, thanks to Prop. \ref{C-In} we also know that $0$ does not lie in $\sigma(\tilde{\Delta}_{\overline{\partial},0,0})$, the spectrum of $$\tilde{\Delta}_{\overline{\partial},0,0}\colon L^2(\tilde{M},\tilde{h})\rightarrow L^2(\tilde{M},\tilde{h})$$ given that it does not lie in the spectrum of $$\tilde{\Delta}\colon L^2(\tilde{M},\tilde{h})\rightarrow L^2(\tilde{M},\tilde{h})$$ and $(M,h)$ is K\"ahler.  Now, using the conjugation and the Hodge star operator, we can conclude that $0$ does not lie in the spectrum of 
$$
\tilde{\Delta}_{\overline{\partial},m,q}\colon L^2\Omega^{m,q}(\tilde{M},\tilde{h})\rightarrow L^2\Omega^{m,q}(\tilde{M},\tilde{h})
$$ 
for each $q=0,\dots,m$. Consequently $0$ does not lie in the spectrum of $$\overline{\eth}_{E,m}\colon L^2\Omega^{m,\bullet}(\tilde{M},\tilde{h})\rightarrow L^2\Omega^{m,\bullet}(\tilde{M},\tilde{h})$$ and thus, by \cite[Prop.7.1.2]{Eys97}, we conclude that there exists a constant  $\epsilon>0$ such that $0$ does not lie in the spectrum of the operator 
\begin{equation}
\label{zaq}
\overline{D}_{m}^s\colon L^2(\tilde{M},\Lambda^{m,\bullet}(\tilde{M})\otimes \mathbb{C})\rightarrow L^2(\tilde{M},\Lambda^{m,\bullet}(\tilde{M})\otimes \mathbb{C})
\end{equation}
 for each $s\in (0,\epsilon)$, see \eqref{L2twist} for a precise definition of the above operator. On the other hand, arguing as in the proof of Th. \ref{nsp}, we can conclude that there exists a constant $\delta>0$ such that $0$ is in the spectrum of \eqref{zaq} for each $s\in (0,\delta)$. We have thus reached a contradiction, as desired. 
\end{proof}

In the case $M$ is a  weakly K\"ahler hyperbolic manifold we get a stronger estimates for its scalar curvature. More precisely:

\begin{thm} 
\label{scalcurv}
Let $M$ be a weakly K\"ahler hyperbolic manifold. Then, for any arbitrarily fixed K\"ahler metric $h$, we have $$\min_{M}(\mathrm{scal}_h)\leq -4\tilde{\lambda}_{0,h}$$ and the equality holds if and only if  $$\mathrm{scal}_h\equiv -4\tilde{\lambda}_{0,h}.$$
\end{thm}

\begin{proof}
Also in this case we proceed by contradiction. We thus assume that the above statement does not hold true and so there exists a K\"ahler metric on $M$ such that $\mathrm{scal}_h\geq -4\tilde{\lambda}_{0,h}$ and $\mathrm{scal}_h(x)> -4\tilde{\lambda}_{0,h}$ for some $x\in M$. Let $\eta$ be an $L^2$ holomorphic $(m,0)$-form on $(\tilde{M},\tilde{h})$. Note that  the Weitzenb\"ock formula on $(\tilde{M},\tilde{h})$ for $\tilde{\Delta}_{\overline{\partial},m,0}$ reads as $$\tilde{\Delta}_{\overline{\partial},m,0}=\frac{1}{2}\left(\nabla^t\circ \nabla+\frac{\mathrm{scal}_{\tilde{h}}}{2}\right)$$ 
with $\nabla$ the connection on $K_{\tilde{M}}$ induced by the Chern connection of $(\tilde{M},\tilde{h})$. Moreover by the fact that $\mathrm{scal}_{\tilde{h}}$ is bounded from below we get that $\nabla \eta\in L^2(\tilde{M},T^*M\otimes K_M)$ and $$\langle \nabla^t(\nabla \eta),\eta \rangle_{L^2\Omega^{m,0}(\tilde{M},\tilde{h})}=\langle\nabla\eta,\nabla\eta\rangle_{L^2(\tilde{M},T^*M\otimes K_M)},$$ see \cite[Th. 2]{VanishingDod}. In this way, keeping in mind the refined Kato inequality \cite[Cor. 4.3]{refKato},  we get 
$$
\begin{aligned}
0&=\langle 2\tilde{\Delta}_{\overline{\partial},m,0}\eta,\eta\rangle_{L^2\Omega^{m,0}(\tilde{M},\tilde{h})}\\
&=\int_{\tilde{M}}\left(|\nabla \eta|^2_{\tilde{h}}+\frac{\mathrm{scal}_{\tilde{h}}}{2}|\eta|^2_{\tilde{h}}\right)\dvol_{\tilde{h}}\\
&\geq \int_{\tilde{M}}\left(2|d|\eta|_{\tilde{h}}|^2_{\tilde{h}}+\frac{\mathrm{scal}_{\tilde{h}}}{2}|\eta|^2_{\tilde{h}}\right)\dvol_{\tilde{h}}\\
(\mathrm{by\ Prop.}\ \ref{C-In}) &\geq \int_{\tilde{M}}\left((2\tilde{\lambda}_{0,h}+\frac{\mathrm{scal}_{\tilde{h}}}{2})|\eta|^2_{\tilde{h}}\right)\dvol_{\tilde{h}}\\
& \geq 0,
\end{aligned}
$$
where, with a little abuse of notation, we have denoted with $|\ |_{\tilde{h}}$ all the Hermitian metrics induced by $\tilde{h}$.

Since $\eta$ is holomorphic and $2\tilde{\lambda}_{0,h}+\frac{\mathrm{scal}_{\tilde{h}}}{2}$ is non-negative and somewhere positive, we can conclude that $\eta$ vanishes on $\tilde{M}$ and consequently $(\tilde{M},\tilde{h})$ carries no non-trivial $L^2$-$(m,0)$ holomorphic forms. On the other hand by \cite[Cor. 3.8]{BDET} we know that $(\tilde{M},\tilde{h})$ carries an infinite dimensional vector space of $L^2$ holomorphic $(m,0)$-forms. We have thus reached a contradiction, as required.
\end{proof}

\begin{rem}
Note that Th. \ref{scalcurv} provides a negative upper bound for the minimum of the scalar curvature of $(M,h)$ that depends only on the isospectral class (in the realm of weakly hyperbolic K\"ahler manifold) of $(M,h)$.
\end{rem}

Arguing as in Cor. \ref{up-es} we have also the following estimates.

\begin{cor}
\label{up-es2}
Let $(M,h)$ be a Kähler topologically hyperbolic manifold and let 
 $c>1$ be the refined Kato constant of $K_M\otimes \Lambda^{1,0}(M)$ w.r.t. $\overline{\partial}_{\Lambda^{1,0}(M),m,0}$. Then 
\begin{enumerate}
\item If $(\Lambda^{1,0}(M),h)$ is Nakano positive over an open subset $A\subset M$ of full measure we have
$$
\min_M\left(\mathrm{scal}_h+2\mathrm{Ric}_h\right)\leq -2c\tilde{\lambda}_{0,h}
$$
 and the equality occurs if and only if 
$$
\mathrm{scal}_h+2r_1\equiv -2c\tilde{\lambda}_{0,h}.
$$
\item If $(\Lambda^{1,0}(M),h)$ is Griffiths positive over an open subset $A\subset M$ of full measure we have
$$
\min_M\left(\mathrm{scal}_h+\mathrm{Ric}_h\right)\leq -c\tilde{\lambda}_{0,h}
$$
 and the equality occurs if and only if 
$$
\mathrm{scal}_h+r_1\equiv -c\tilde{\lambda}_{0,h}.
$$
\end{enumerate}
\end{cor}

\begin{proof}
By comparing the Weitzenb\"ock formula for the Hodge Laplacian acting on one form with the Akizuki--Nakano identity we get $$[i\Theta(\Lambda^{1,0}(M)),\Lambda]=\mathrm{ric}^{1,0}$$ with $\mathrm{ric}^{1,0}$ defined in \eqref{ric10}. The conclusion follows now by Cor. \ref{up-es}. 

If $(\Lambda^{1,0}(M),h)$ is Griffiths positive over $A\subset M$ then $(K_M\otimes\Lambda^{1,0}(M))$ is Nakano positive over $A$, see \cite{Dema-Skoda}. In this case we have 
$$
\bigl[i\Theta\bigl(K_M\otimes\Lambda^{1,0}(M)\bigr),\Lambda\bigr]=\frac{\mathrm{scal}_h}{2}+\mathrm{ric}_h
$$ 
and now the conclusion follows again by  Cor. \ref{up-es}.
\end{proof}

We collect now some consequences that follow easily from Prop. \ref{curvature1} and Th. \ref{scalcurv}.

\begin{cor}\label{cor:noKm}
Let $M$ be a weakly K\"ahler hyperbolic manifold. Then, $M$ carries no K\"ahler metric with $\mathrm{scal}_h>-4\tilde{\lambda}_{0,h}$.
\end{cor}

\begin{rem}\label{rem:bdpp}
We know by \cite{BDET} that a weakly K\"ahler hyperbolic manifold is projective of general type. In particular, it is not uniruled nor Calabi--Yau. As remarked in \cite[Theorem 1.4]{GW12}, a consequence of \cite{BDPP13} is that every K\"ahler metric on $M$ needs to have negative total scalar curvature. We can thus interpret the above corollary as a quantitative information of this fact in the special case of weakly K\"ahler hyperbolic manifolds: not only one cannot have K\"ahler metrics whose scalar curvature average is non-negative but also point-wise the scalar curvature has to become somewhere negative enough in a precise sense. 
\end{rem}

\begin{cor}
Let $M$ be a K\"ahler topologically hyperbolic manifold of complex dimension $m$, and let $h$ be an arbitrarily fixed K\"ahler metric on $M$. If $a\in \mathbb{R}$ verifies $\mathrm{Ric}_h\geq a$ then $$a\leq -\frac{\tilde{\lambda}_{0,h}}{m}.$$
\end{cor}

\begin{proof}
If $a>-\frac{\tilde{\lambda}_{0,h}}{m}$ then $\mathrm{Ric}_h> -\frac{\tilde{\lambda}_{0,h}}{m}$ and consequently $$\mathrm{ric}^{p,0}>-\frac{p}{m}\tilde{\lambda}_{0,h}\geq -\tilde{\lambda}_{0,h}$$ for each $p=1,\dots,m$. This contradicts Prop. \ref{curvature1}.
\end{proof}

If the Ricci curvature is suitably negative we get further interesting geometric consequences. More precisely:

\begin{thm}
\label{Ricci2}
Let $M$ be a Kähler topologically hyperbolic manifold of complex dimension $m$. If there exists a K\"ahler metric $h$ such that for each $p\in\{1,\dots,m-1\}$ the inequality $$\mathrm{ric}^{p,0}> -\tilde{\lambda}_{0,h}$$ holds true over an open subset of $M$ of full measure  then $$\chi(M,K_M)=h^{m,0}_{(2),\overline{\partial}}(M)>0.$$ If in addition $M$ has generically large fundamental group then $K_M$ is big and thus $M$ is projective. 
\end{thm}

\begin{proof}
 Arguing as in the proof of Th. \ref{curvature1} we can conclude that $0$ does not lie in the spectrum of $\tilde{\Delta}_{\overline{\partial},m,q}\colon L^2\Omega^{m,q}(\tilde{M},\tilde{h})\rightarrow L^2\Omega^{m,q}(\tilde{M},\tilde{h})$ for each $q=1,\dots,m$. On the other hand, by adopting the same strategy used in the proof of Th. \ref{nsp}, we know that $0$ lies in the spectrum of 
 $$
 \overline{\eth}_{m}\colon L^2\Omega^{m,\bullet}(\tilde{M},\tilde{h})\rightarrow L^2\Omega^{m,\bullet}(\tilde{M},\tilde{h})
 $$ 
and consequently $0$ lies in the spectrum of 
$$
\tilde{\Delta}_{\overline{\partial},m,0}\colon L^2\Omega^{m,0}(\tilde{M},\tilde{h})\rightarrow L^2\Omega^{m,0}(\tilde{M},\tilde{h}).
$$ 
Note now that since $0\notin \sigma(\tilde{\Delta}_{\overline{\partial},m,1})$ we know that $\im(\tilde{\Delta}_{\overline{\partial},m,1})$ is closed in $L^2\Omega^{m,1}(\tilde{M},\tilde{h})$. Consequently $\overline{\partial}_{m,0}\colon L^2\Omega^{m,0}(\tilde{M},\tilde{h})\rightarrow L^2\Omega^{m,1}(\tilde{M},\tilde{h})$ has closed range and this in turn implies that its adjoint $\overline{\partial}_{m,0}^*\colon L^2\Omega^{m,1}(\tilde{M},\tilde{h})\rightarrow L^2\Omega^{m,1}(\tilde{M},\tilde{h})$ has closed range, too. 

Finally since both $\overline{\partial}_{m,0}$ and $\overline{\partial}_{m,0}^*$ have closed range we can conclude that $\overline{\partial}_{m,0}^*\circ \overline{\partial}_{m,0}\colon L^2\Omega^{m,0}(\tilde{M},\tilde{h})\rightarrow L^2\Omega^{m,0}(\tilde{M},\tilde{h})$ has closed range, that is $\im(\tilde{\Delta}_{\overline{\partial},m,0})$ is closed in $L^2\Omega^{m,0}(\tilde{M},\tilde{h})$. Now, since  $\im(\tilde{\Delta}_{\overline{\partial},m,0})$ is closed and $0\in \sigma(\tilde{\Delta}_{\overline{\partial},m,0})$, we can deduce eventually that $\ker(\tilde{\Delta}_{\overline{\partial},m,0})\neq \{0\}$ and thus $h^{m,0}_{(2),\overline{\partial}}(M)>0$. The conclusion now follows by Atiyah's $L^2$-index theorem. Namely $$h^{m,0}_{(2),\overline{\partial}}(M)=\chi(M,K_M)>0$$ as required. 

Assume now that $M$ has generically large fundamental group. Since we showed above that $(\tilde{M},\tilde{h})$ carries non trivial $L^2$-holomorphic $(m,0)$-forms we can apply \cite[Cor. 13.10]{Kollarbook} to conclude that $K_M$ is big. 

Finally since $K_M$ is big and $(M,h)$ is K\"ahler we can conclude that $M$ is projective.
\end{proof}

\begin{cor}
Let $M$ be Kähler topologically hyperbolic manifold of complex dimension $m>1$. If there exists a K\"ahler metric $h$ such that $$\mathrm{Ric}_h> -\frac{\tilde{\lambda}_{0,h}}{m-1}$$  then $$\chi(M,K_M)=h^{m,0}_{(2),\overline{\partial}}(M)>0.$$ If in addition $M$ has generically large fundamental group then $K_M$ is big and thus $M$ is projective. 
\end{cor}

\begin{proof}
As in \eqref{eigenricci} let us denote with  $\{r_1(p),r_1(p),\dots,r_{m}(p),r_{m}(p)\}$ the eigenvalues of $\mathrm{ric}_p\colon T_pM\rightarrow T_pM$. Then, as previously observed, the eigenvalues of $\mathrm{ric}_p^{1,0}\colon T_p^{1,0}M\rightarrow T_p^{1,0}M$ are $\{r_1(p),r_2(p),\dots,r_{m}(p)\}$. 

Since $\mathrm{ric}_p^{s,0}\colon \Lambda^{s,0}(M)\rightarrow \Lambda^{s,0}(M)$ is the extension of $\mathrm{ric}_p^{1,0}$ as derivation,we know that the eigenvalues of $\mathrm{ric}^{s,0}_p$ are given by all the combinations $$r_{j_1}(p)+\cdots+r_{j_s}(p)$$ with $j_1,\dots,j_s\in \{1,\dots,m\}$ and $j_1<j_2<\cdots<j_s$. Now it is clear that if  $\mathrm{Ric}_h> -\tilde{\lambda}_{0,h}/(m-1)$ then $\mathrm{ric}_p^{s,0}>-\tilde{\lambda}_{0,h}$ for each $s=1,\dots,m-1$. The conclusion follows by Th. \ref{Ricci2}.  
\end{proof}

\newpage
\appendix

\section{An intrinsic characterization of hyperbolic classes, by Benoît Claudon\protect\footnote{Benoît Claudon, Univ. Rennes, CNRS, IRMAR - UMR 6625, F-35000 Rennes, France. \emph{E-mail}: \texttt{benoit.claudon@univ-rennes.fr}. \newline The author benefits from the support of the French government \lq\lq Investissements d'Avenir\rq\rq{} program integrated to France 2030, bearing the following reference ANR-11-LABX-0020-01.}}\label{sec:appendix}

\subsection*{A layman realization of the Hurewicz morphism}
	
Let $G$ be a finitely presented group. We are interested in the second cohomology group of $G$, namely
\[
H^2(G,\R):=H^2\left(C^\bullet(G,\R),d\right)\quad \text{where}\ C^k(G,\R)=\{c\colon G^k\to \R\}
\]
together with the differential
\begin{multline}\label{eq:differential}
\begin{gathered}
(d\,c)(g_1,\ldots,g_{k+1}):=c(g_2,\ldots,g_{k+1}) \\ + \sum_{i=1}^k (-1)^ic(g_1,\ldots,g_{i-1},g_ig_{i+1},g_{i+2},\ldots,g_{k+1}) + (-1)^{k+1}c(g_1,\ldots,g_k).
\end{gathered}
\end{multline}
In degree~$2$, we have thus
\[
H^2(G,\R)=\frac{Z^2(G,\R)}{dC^1(G,\R)}
\]
where
\begin{align*}
Z^2(G,\R)&:=\big\{c\colon G^2\to \R\mid \forall\,(g,h,k)\in G^3,\\
&\qquad\ c(g,h)=c(h,k)-c(gh,k)+c(g,hk)\big\}
\end{align*}
and
\begin{align*}
dC^1(G,\R)&:=\big\{c\colon G^2\to \R\mid \exists\,f\colon G\to \R,\\
&\qquad \forall\, (g,h)\in\R^2,\ c(g,h)=f(gh)-f(g)-f(h)\big\}.
\end{align*}

 It is well known that $G$ can be realized as the fundamental group of a closed manifold. Let $X$ be such a manifold and $\pi_X:\tilde{X}\to X$ its universal covering. According to a result of Hopf \cite{Hopf42}, the vector space $H^2(G,\R)$ is finite dimensional and it is canonically identified \emph{via} the Hurewicz morphism with the kernel of the map induced by $\pi_X$. In other words, the sequence
\[
0 \To H^2(G,\R) \To H^2(X,\R) \overset{\pi_X^*}{\To} H^2(\tilde{X},\R)
\]
is exact.

Let us first explain how a $2$-form in the kernel of $\pi_X^*$ gives rise to a $2$-cocyle of the group $G$. To do so, let us consider $\omega$ a closed $2$-form on $X$ such that its pull-back to the universal covering $\tilde{X}$ of $X$ can be written $\tilde{\omega}=d\alpha$ (with $\alpha$ a $1$-form on $\tilde{X}$). Using the action of $G=\pi_1(X)$ onto $\tilde{X}$, we have
$$
\forall\,g\in G,\ d(g^*\alpha)=g^*d\alpha=g^*\omega=\omega=d(\alpha)
$$
and the $1$-form $g^*\alpha-\alpha$ is closed on $\tilde{X}$. The latter being simply connected, we can write $g^*\alpha-\alpha=df_g$ where $f_g$ is a smooth function, normalized by the condition $f_g(x_0)=0$ (where $x_0$ is any fixed base-point on $\tilde{X}$). Using the group law, we can also write
$$
\forall\, (g,h)\in G^2,\ df_{gh}=(gh)^*\alpha-\alpha=h^*\left(g^*\alpha-\alpha\right)+h^*\alpha-\alpha=h^*df_g+df_h.
$$
The function $c(g,h):=f_{gh}-h^*f_g-f_h$ is thus a real constant: $c(g,h)\in\R$. Since it will be useful in the sequel, we can give a more concise formula for this constant by evaluating the previous equality at $x_0$ (the base-point); it yields
\begin{equation}\label{eq:class in terms of the primitive}
\forall\,(g,h)\in G^2,\ c(g,h)=-f_g(h^{-1}(x_0)).
\end{equation}
Using the formal expression $c(g,h)=f_{gh}-h^*f_g-f_h$, it is easy to see that the cochain $c$ is indeed a cocyle: it defines thus a degree~$2$ cohomology class $[c]\in H^2(G,\R)$.

\medskip

Let us now describe how the degree~$2$ classes coming from the group $G$ can be realized as differential forms. Let $c\in C^2(G,\R)$ satisfying the cocycle identity:
$$
\forall\, (g,h,k)\in G^3,\ c(g,h)=c(h,k)-c(gh,k)+c(g,hk).
$$
We can rewrite it in the following form:
\begin{equation}\label{eq:cocyle}
\forall\,  (g,h,\gamma)\in G^3,\ c(\gamma g,g^{-1}h)=c(g,g^{-1}h)-c(\gamma^{-1},\gamma h)+c(\gamma^{-1},\gamma g).
\end{equation}
We can cook up an exact $2$-form $\widetilde{\omega}$ on $\tilde{X}$ that is invariant under the action of $G=\pi_1(X)$. In this way, we shall get that $\widetilde{\omega}=\pi_X^*(\omega)$ and $\omega$ (its de Rahm cohomology class) will be the sought closed $2$-form on $X$.

To do so, let us fix a finite open cover $\{V_i\}_{i\in I}$ of $X$ with $V_i$ simply connected and let us pick one connected component $U_i$ of $p^{-1}(V_i)$ ($V_i$ being simply connected $p$ induces an isomorphism between $U_i$ and $V_i$) such that
$$
p^{-1}(V_i)=\bigcup_{g\in G}gU_i.
$$
Finally, let us denote by $\{\chi_i\}_{i\in I}$ a partition of unity subordinate to the open cover $\{V_i\}_{i\in I}$; it induces a partition of unity on $\tilde{X}$, to be denoted by $\{\chi_{i,g}\}_{(i,g)\in I\times G}$, subordinated to the open cover $\{gU_i\}_{(i,g)\in I\times G}$. The function $\chi_{i,g}$ is supported in $gU_i$ and is nothing but $\chi_{i}\circ p|_{gU_i}$. This family enjoys the obvious equivariance property
\begin{equation}\label{eq:equivariance-partition}
\forall\, (\gamma,g)\in G^2,\ \forall\,i\in I,\ \gamma^*\chi_{i,g}=\chi_{i,\gamma^{-1} g}.
\end{equation}
With this at hand, let us consider
\begin{equation}\label{eq:from defined by cocyle}
\begin{gathered}
\tilde{\omega}:=\sum_{\substack{(i,j)\in I^2 \\ (g,h)\in G^2}} c(g,g^{-1}h)\,d\chi_{i,g}\wedge d\chi_{j,h}\\
=d\biggl(\sum_{\substack{(i,j)\in I^2 \\ (g,h)\in G^2}} c(g,g^{-1}h)\,\chi_{i,g}\wedge d\chi_{j,h}\biggr)=d\alpha.
\end{gathered}
\end{equation}
From the last equality we see that $\tilde{\omega}$ is exact. Now, let us check that $\tilde{\omega}$ is invariant under the action of $\gamma\in G$. We have
\begin{align}
\gamma^*\tilde{\omega}&=\sum_{\substack{(i,j)\in I^2 \\ (g,h)\in G^2}} c(g,g^{-1}h)\,\gamma^*\left(d\chi_{i,g}\wedge d\chi_{j,h}\right) \nonumber \\
&=\sum_{\substack{(i,j)\in I^2 \\ (g,h)\in G^2}} c(g,g^{-1}h)\,d\chi_{i,\gamma^{-1} g}\wedge d\chi_{j,\gamma^{-1} h} \label{eq:computation0} \\
&=\sum_{\substack{(i,j)\in I^2 \\ (g,h)\in G^2}} c(\gamma g,g^{-1}h)\,d\chi_{i,g}\wedge d\chi_{j,h} \label{eq:computation1} \\
&=\sum_{\substack{(i,j)\in I^2 \\ (g,h)\in G^2}} \bigl(c(g,g^{-1}h)-c(\gamma^{-1},\gamma h)+c(\gamma^{-1},\gamma g)\bigr)\,d\chi_{i,g}\wedge d\chi_{j,h} \label{eq:computation2} \\
&=\tilde{\omega}. \nonumber
\end{align}
Indeed, the equality \eqref{eq:computation0} is just the equivariance property \eqref{eq:equivariance-partition}, the equality \eqref{eq:computation1} is a change of variables, and the equality \eqref{eq:computation2} is nothing but the cocycle property \eqref{eq:cocyle}. To see that the remaining terms of the sum vanish, it suffices to observe that we can rewrite the corresponding sum as
\begin{multline*}
\sum_{\substack{(i,j)\in I^2 \\ (g,h)\in G^2}} c(\gamma^{-1},\gamma h)\,d\chi_{i,g}\wedge d\chi_{j,h} \\
=-\sum_{(j,h)\in I\times G}c(\gamma^{-1},\gamma h)\,d\chi_{j,h}\wedge d\biggl(\sum_{(i,g)\in I\times G} \chi_{i,g}\biggr)=0
\end{multline*}
since the sum in parenthesis is identically equal to $1$.

\subsection*{Hyperbolic classes}

We now come to the notion of hyperbolic classes. Let us first recall their definition.

\begin{defn}\label{def:hyperbolic class}
Let $X$ be a compact manifold and $[\omega]\in H^2(X,\R)$. The class $[\omega]$ is said to be \emph{hyperbolic} if $\pi_X^*\omega$ admits a bounded primitive on $\tilde{X}$. We denote by $H^2_{\hyp}(X,\R)$ the corresponding subspace of $H^2(X,\R)$.
\end{defn}

\begin{rem}
By compactness of $X$, this definition is independent of the representative $\omega$ of the class and of the implicit Riemannian metric chosen to measure the norm of the differential forms (provided it comes from the compact base).
\end{rem}

By definition, a hyperbolic class comes from the group: $\omega\in H^2(G,\R)$ where $G=\pi_1(X)$. Since the condition above refers to a bounded primitive on $\tilde{X}$ it is not obvious that the subspace it defines is independent of the realization of the group $G$ as the fundamental group of a compact manifold. It was proven by Brunnbauer, Kotschick, and Schönlinner in \cite[Theorem 2.4]{BKS24} (see also \cite[Theorem~5.1]{Kedra09}). Here we give an explicit description of the subspace $H^2_{\hyp}(G,\R)$.

\begin{prop}\label{prop:explicit description}
Let us consider the following subspace defined in terms of cochains $C^2(G,\R)$:
\[
H^2_{\hyp}(G,\R):=\left\{[c] \in H^2(G,\R) \mid \forall\, h\in G,\ c(\bullet,h)\in\Linf{G,\R}\right\}.
\]
If $X$ is any compact manifold having $G$ as fundamental group, the Hurewicz morphism induces an isomorphism
\[
H^2_{\hyp}(G,\R)\simeq H^2_{\hyp}(X,\R).
\]
\end{prop}

\begin{proof}
Let $[\omega]\in H^2_{\hyp}(X,\R)$. With the notation used above, we get that $\tilde{\omega}=d\alpha$ with $\alpha$ a bounded $1$-form. In particular, the family of functions $f_g$ ($g\in G$) such that $g^*\alpha-\alpha=df_g$ is uniformly Lipschitz: $\Lip(f_g)\le 2|\alpha|_{\infty}$. From \eqref{eq:class in terms of the primitive}, we have
\[
| c(g,h)|=|f_g(h^{-1}(x_0))|=|f_g(h^{-1}(x_0))-f_g(x_0)|\le 2|\alpha|_{\infty}d(x_0,h(x_0))
\]
where $d$ is the distance induced on $\tilde{X}$ by any Riemannian metric on $X$ (the elements of $G$ act then by isometries on $(\tilde{X},d)$). We thus get the sought boundedness condition.

\medskip
In the reverse way, let $c=(c(g,h))_{(g,h)\in G^2}$ be a cocycle with $c(\bullet,h)\in\Linf{G,\R}$ for every $h\in G$. As in the proof of \cite[Lemma~2.10]{BCDT}, we just have to check that the form $\tilde{\omega}$ cooked up in~\eqref{eq:from defined by cocyle} has a bounded primitive. The primitive given by

\begin{align*}
\alpha &:=\sum_{\substack{(i,j)\in I^2 \\ (g,h)\in G^2}} c(g,g^{-1}h)\,\chi_{i,g}\wedge d\chi_{j,h} \\
 &=\sum_{\substack{(i,j)\in I^2 \\ (g,k)\in G^2}} c(g,k)\,\chi_{i,g}\wedge d\chi_{j,gk}
\end{align*}
has the desired property. Indeed, since $\Supp{\chi_{i,g}}=g(U_i)$, the product $\chi_{i,g}\wedge d\chi_{j,gk}$ does not vanish if $U_i\cap k(U_j)\neq\emptyset$. The action of $G$ on $\tilde{X}$ being proper, the set
\[
K:=\{k\in G\mid \exists (i,j)\in I^2,\ U_i\cap k(U_j)\neq\emptyset\}
\]
is finite and we can consider
\[
C:=\max(|c(\bullet,k)|_{\infty}\mid k\in K)<+\infty.
\]
With this notation, we have
\[
|\alpha|_{\infty}\le C |I| |K| \max(|d\chi_{i}|,\, i\in I)
\]
and the form $\tilde{\omega}$ has a bounded primitive.
\end{proof}

\subsection*{A cohomological reformulation}

The condition of being bounded when the second variable is fixed as been recently studied in relation with the so-called bounded and $\linf$-cohomology\footnote{These classes are called \emph{weakly bounded} in \cite{FS23,AM24}. It seems that this terminology originates in the article \cite{NR97} where degree~2 classes with integral coefficients were studied in terms of central extensions.}. Let us give a brief account of these theories. The bounded cohomology of a (discrete) group $G$ is the cohomology of the complex
\[C^\bullet_b(G,\R):=\left\{c\in C^\bullet(G,\R)\mid c\ \text{is bounded}\right\}\]
endowed with the same differential~\eqref{eq:differential}. The inclusion of sub-complex
\[C^\bullet_b(G,\R)\subset C^\bullet(G,\R)\]
gives rise to a comparison morphism:
\[\compb^\bullet\colon H^\bullet_b(G,\R)\To H^\bullet (G,\R).\]
In the sequel, we will denote its image as
\begin{equation}\label{eq:image-bounded-coh}
bH^\bullet(G,\R):=\compb^\bullet\left( H^\bullet_b(G,\R) \right)\subset H^\bullet (G,\R).
\end{equation}
For a comprehensive introduction to bounded cohomology, the reader is referred to the book \cite{FrigerioBook}.

The $\linf$-cohomology has been introduced by Gersten in an unpublished preprint entitled \emph{Bounded cocycles and combings of groups}; hopefully, some parts of the content of this preprint can be found in \cite{Gersten95}.\footnote{It is a surprising fact but the published paper \cite{Gersten92} has the same title as the above-mentioned preprint but it does not deal with $\linf$-cohomology.} The preprints \cite{Milizia1,Milizia2} give also a nice exposition of this theory from a geometric viewpoint. The $\linf$-cohomology is defined as the cohomology with values in a non-trivial module (see \cite{brown-book} for this general framework). Given a discrete group $G$, the Banach space $\Linf{G,\R}$ is a left $G$-module under the action:
\[\forall\,(g,h)\in G^2,\ \forall\, \alpha\in\Linf{G,\R},\ g\cdot\alpha (h):=\alpha(g^{-1}h).\]
The $\linf$-cohomology is then defined as
\[H^\bullet_{(\infty)}(G,\R):=H^\bullet \left(C^\bullet \left(G,\Linf{G,\R}\right),d \right),\]
where $d$ is a differential defined in a way similar to~\eqref{eq:differential} but taking the $G$ action on $\Linf{G,\R}$ into account.\footnote{The original definition given by Gersten was expressed in terms of $K(G,1)$ satisfying some finiteness conditions; this algebraic definition of the $\linf$-cohomology appears in \cite{Wienhard12}.} The inclusion of the constant (bounded) functions
\[\R\subset \Linf{G,\R}\]
induces another comparison morphism
\[\compinfty^\bullet\colon H^\bullet(G,\R)\To H^\bullet_{(\infty)}(G,\R).\]

These two comparison morphisms are compatible in quite a strong way.
\begin{prop}[\emph{cf.} \protect{\cite[Remark, p.~91]{Gersten95}}]\label{prop:compatibility-comparison-maps}
For any degree $k\ge 1$, the composition
\[H^k_b(G,\R)\xrightarrow{\compb^k}H^k(G,\R)\xrightarrow{\compinfty^k}H^k_{(\infty)}(G,\R)\]
is the zero map.
\end{prop}

The description of hyperbolic classes given above combined with \cite[Proposition~1.11]{FS23} yields the following statement.
\begin{prop}\label{prop:hyperbolic=kernel}
For any finitely presented group $G$, the subspace $H^2_{\hyp}(G,\R)$ is nothing but the kernel of the map $\compinfty^2$:
\[H^2_{\hyp}(G,\R)=\ker\left(\compinfty^2\colon H^2(G,\R) \To H^2_{(\infty)}(G,\R)\right).\]
\end{prop}

\begin{rem}
Propositions~\ref{prop:compatibility-comparison-maps} and~\ref{prop:hyperbolic=kernel} give back the inclusion
\[bH^2(G,\R)\subset H^2_{\hyp}(G,\R).\]
It should be noted that this inclusion is not an equality in general. Examples of groups $G$ with strict inclusion
\[bH^2(G,\R)\subsetneq H^2_{\hyp}(G,\R)\]
can be found in \cite[\S~4]{FS23} and \cite[\S~3]{AM24}: the example in the former is only finitely generated but the latter reference provides us with a finitely presented group together with a degree~$2$ hyperbolic class which is not bounded.
\end{rem}

\subsection*{Basic properties of hyperbolic classes}

With the characterization of hyperbolic classes given above (Proposition~\ref{prop:explicit description}), the following result is an immediate consequence of the existence of a right-invariant mean value on $\Linf{G,\R}$ for amenable groups.
\begin{prop}[\emph{cf.}~\protect{\cite[Proposition~5.1]{FS23}}]\label{prop:vanishing amenable}
If $G$ is an amenable group, we have
\[
H^2_{\hyp}(G,\R)=0.
\]
\end{prop}
The latter result can be found in various place of the literature in other geometric contexts \cite{Kedra09,BKS24,BDET}.

Before stating the next result, let us note that the hyperbolic classes are preserved by pull-back: if $\varphi\colon G\to Q$ is a morphism between finitely presentable groups, the class $\varphi^*\alpha$ is hyperbolic if $\alpha\in H^2_{\hyp}(G,\R)$ (this is obvious in view of Proposition~\ref{prop:explicit description}). We now investigate the behavior of the hyperbolic classes with respect to a product of groups.
\begin{prop}\label{prop:hyp product}
Let $G_1$ and $G_2$ be two finitely presentable groups and let us denote $p_1$ and $p_2$ the corresponding projections. The map $p_1^*\oplus p_2^*$ induces an isomorphism
\[H^2_{\hyp}(G_1,\R)\oplus H^2_{\hyp}(G,\R)\simeq H^2_{\hyp}(G_1\times G_2,\R).\]
\end{prop}
\begin{proof}
The arguments below are essentially contained in the proof of \cite[Proposition~5.17]{FS23}. Let us consider $j_1\colon G_1\to G_1\times G_2$ (resp.~$j_2$) the natural injection of $G_1$ (resp.~$G_2$). If $\alpha\in H^2_{\hyp}(G_1\times G_2,\R)$, let us consider $\tilde{\alpha}:=\alpha-p_1^*\left(j_1^*\alpha\right)-p_2^*\left(j_2^*\alpha\right)$. By the observation above, the class $\tilde{\alpha}$ is still hyperbolic. We want to prove that $\tilde{\alpha}=0$ and, to do so, we look at the K\"unneth decomposition of the homology of the products:
\[H_2(G_1\times G_2,\R)\simeq H_2(G_1,\R)\oplus H_1(G_1,\R)\otimes H_1(G_2,\R)\oplus H_2(G_2,\R).\]
We shall prove that $\langle\tilde{\alpha},\beta\rangle =0$ for any $\beta\in H_2(G_1\times G_2,\R)$. If $\beta=(j_1)_*\beta_1$ with $\beta_1\in H_2(G_1,\R)$, we then have:
\begin{align*}
\langle\tilde{\alpha},(j_1)_*\beta_1\rangle &= \langle \alpha,(j_1)_*\beta_1\rangle - \langle p_1^*\left(j_1^*\alpha\right) ,(j_1)_*\beta_1\rangle - \langle p_2^*\left(j_2^*\alpha\right) , (j_1)_*\beta_1\rangle \\
&= \langle j_1^* \alpha,\beta_1\rangle - \langle j_1^*p_1^*\left(j_1^*\alpha\right) ,\beta_1\rangle - \langle j_1^*p_2^*\left(j_2^*\alpha\right) , \beta_1\rangle \\
&=\langle j_1^* \alpha,\beta_1\rangle - \langle j_1^*\alpha ,\beta_1\rangle - \langle 0 , \beta_1\rangle =0
\end{align*}
and the same computation applies if $\beta=(j_2)_*\beta_2$ with $\beta_2\in H_2(G_2,\R)$. We are left to consider the case when $\beta\in H_1(G_1,\R)\otimes H_1(G_2,\R)$ and, by linearity, we are reduced to the case $\beta=\gamma_1\otimes\gamma_2$ with $\gamma_i\in H_1(G_i,\R)$ (for $i=1,2$). Each class $\gamma_i$ being induced by a morphism $\varphi_i\colon \Z\to G_i$, the class $\beta$ is actually induced by a morphism $\varphi\colon \Z^2\to G_1\times G_2$ and we can write $\beta=\varphi_*\gamma$ with $\gamma\in H_2(\Z^2,\R)$. We then have:
\[\langle \tilde{\alpha},\beta\rangle = \langle \tilde{\alpha},\varphi_*\gamma\rangle = \langle \varphi^*\tilde{\alpha},\gamma\rangle = \langle 0,\gamma\rangle =0.\]
The group $\Z^2$ is indeed amenable and Proposition~\ref{prop:vanishing amenable} (plus the fact that $\tilde{\alpha}$ is hyperbolic) implies that $\varphi^*\tilde{\alpha}=0$. As mentioned above, we finally have $\tilde{\alpha}=0$ and $\alpha=p_1^*\left(j_1^*\alpha\right)+p_2^*\left(j_2^*\alpha\right)$ as desired.
\end{proof}

\begin{quest}
In particular, if $G_2$ is amenable, we get that $H^2_{\hyp}(G_1\times G_2,\R)\simeq H^2_{\hyp}(G_1,\R)$. It is then natural ask if the corresponding statement holds in the more general following setting: let $p\colon G\twoheadrightarrow Q$ be a surjective morphism of finitely presented groups such that $\ker(p)$ is amenable; is it true that $p$ induces a surjective map $p^*\colon H^2_{\hyp}(Q,\R)\twoheadrightarrow H^2_{\hyp}(G,\R)$? This is know to hold for bounded classes ($p^*$ is even an isomorphism in this case, see for instance \cite[Corollary~4.25]{FrigerioBook}).
\end{quest}

\subsection*{K\"ahler topologically hyperbolic surfaces}

We end this note by showing the following statement.
\begin{thm}\label{prop:general type surface}
Let $X$ be a compact K\"ahler surface and assume that $X$ is K\"ahler topologically hyperbolic. Then $X$ is of general type: $\kappa(X)=2$.
\end{thm}
\begin{proof}
We explore the Kodaira classification in the K\"ahler topologically hyperbolic case. From Corollary \ref{cor:pseff} and Theorem \ref{thm:c1} of the present paper, we know that the Kodaira dimension of a K\"ahler topologically hyperbolic surface $X$ has to satisfy $\kappa(X)\ge1$. If $\kappa(X)=1$, then $X$ is the total space of an elliptic fibration $f\colon X\to C$ over a curve. The fundamental group of $X$ can be described according to the following alternative:
\begin{itemize}
\item[---] either $\pi_1(E)_X:=\im(\pi_1(E)\to\pi_1(X))=1$ (where $E$ is a general smooth fiber of $f$) and, up to replacing $X$ with a finite \'etale covering, $f_*\colon \pi_1(X)\simeq\pi_1(C)$;
\item[---] or $\pi_1(E)$ injects into $\pi_1(X)$ and, up to a finite \'etale covering, $X$ is a product $E\times C$.
\end{itemize}
In both cases, we get that $H^2_{\hyp}(\pi_1(X),\R)=H^2_{\hyp}(\pi_1(C),\R)$ and in particular $\alpha^2=0$ for any hyperbolic class on $X$. A compact K\"ahler surface with $\kappa(X)<2$ cannot thus be K\"ahler topologically hyperbolic and we are done.
\end{proof}

\subsubsection*{Acknowledgments}
I would like to warmly thank Pierre~Py for reading a draft of this note, for drawing my attention to the reference \cite{FS23} and for enlightening discussions on the various cohomology theories of groups.

\bibliography{bibliography}{}

\end{document}